\documentclass[a4paper, 11pt]{article}

\usepackage{amsmath}
\usepackage{graphicx,psfrag}
\usepackage{amsfonts}
\usepackage{dsfont}
\usepackage{bm}

\newtheorem{thm}{Theorem}[section]
\newtheorem{cor}[thm]{Corollary}
\newtheorem{lemma}[thm]{Lemma}
\newtheorem{prop}[thm]{Proposition}
\newtheorem{remark}[thm]{Remark}

\newtheorem{defn}[thm]{Definition}

\newtheorem{example}[thm]{Example}

\newenvironment{proof}{{\bf Proof:\,}}{\hspace*{\fill}\rule{1.2ex}{1.2ex}\\ }

\makeatletter

\@addtoreset{equation}{section}
\makeatother

\newcommand{\N}{\mathds{N}}

\DeclareMathOperator*{\essi}{ess\,inf}

\newcommand{\R}{\mathds{R}}
\newcommand{\Nabla}{\bm{\nabla}}

\begin{document}

\title{Heat Flow on Finsler Manifolds}
\author{Shin-ichi Ohta\thanks{
Partly supported by the JSPS fellowship for research abroad and
the Grant-in-Aid for Young Scientists (B) 20740036.}, Karl-Theodor Sturm}

\maketitle

\section*{}

This paper studies the heat flow on Finsler manifolds.
A \emph{Finsler manifold} is a smooth manifold $M$ equipped with a \emph{Minkowski norm}
$F(x,\cdot ):T_xM \rightarrow \R_+$ on each tangent space.
Mostly, we will require that this norm is strongly convex and smooth and that it depends smoothly on the base point $x$.
The particular case of a \emph{Hilbert norm} on each tangent space leads to the important subclasses of Riemannian manifolds
where the heat flow is widely studied and well understood.
We present two approaches to the heat flow on a Finsler manifold:
\begin{itemize}\item
either as gradient flow on $L^2(M,m)$ for the \emph{energy}
\[ \mathcal{E}(u) =\frac{1}{2} \int_M F^2(\Nabla u) \,dm; \]
\item
or as gradient flow on the {\it reverse} $L^2$-Wasserstein space $\mathcal{P}_2(M)$
of probability measures on $M$ for the \emph{relative entropy}
\[ \mathrm{Ent}(u) =\int_M u \log u \,dm. \]
\end{itemize}
Both approaches depend on the choice of a measure $m$ on $M$ and then lead to the same \emph{nonlinear} evolution semigroup.
We prove $\mathcal{C}^{1,\alpha}$-regularity for solutions to the (nonlinear) heat equation on the Finsler space $(M,F,m)$.
Typically, solutions to the heat equation will not be $\mathcal{C}^2$.
Moreover, we derive pointwise comparison results \'a la Cheeger-Yau and
integrated upper Gaussian estimates \'a la Davies.

\section{Finsler Manifolds}\label{sc:pre}

\subsection{Finsler Structures}

Throughout this paper,  a \emph{Finsler manifold} will be a pair  $(M,F)$ where $M$ is a smooth, connected $n$-dimensional manifold
and $F:TM \rightarrow \R_+$ is a measurable function (called \emph{Finsler structure}) with the following properties:
\begin{itemize}
\item[(i)] $F(x,c\xi) =c F(x,\xi)$ for all $(x,\xi) \in TM$ and all $c>0$.
\item[(ii)] For each point $\overline{x} \in M$,
there are a local coordinate system $(x^i)_{i=1}^n$ on a neighborhood $U$ of $\overline{x}$
and positive numbers $\lambda$ and $\lambda^*$ such that, for almost every $x \in U$, the function
$F^2(x,\cdot)$ on $T_xM \setminus \{0\}$ is twice differentiable and the $(n \times n)$-matrix
\begin{equation}\label{eq:F-str1}
g_{ij}(x,\xi):= \frac{\partial^2}{\partial \xi^i \partial \xi^j} \bigg( \frac{1}{2} F^2(x,\xi) \bigg)
\end{equation}
is uniformly elliptic on $U$ in the sense that
\begin{equation}\label{eq:F-str2}
\lambda^* \sum_{i=1}^n (\eta^i)^2 \le \sum_{i,j=1}^n g_{ij}(x,\xi) \eta^i \eta^j
 \le \frac1{\lambda} \sum_{i=1}^n (\eta^i)^2
\end{equation}
holds for all $\xi \in T_xM \setminus \{0\}$ and all $\eta \in T_xM$.
Here $(x^i,\xi^i)_{i=1}^n$ denotes the local coordinate system on $\pi^{-1} (U) \subset TM$ given by
$\xi=\sum_{i=1}^n \xi^i (\partial/\partial x^i)$.
We will say that such a point $x$ is {\it regular}.
\end{itemize}

The uniform ellipticity $(\ref{eq:F-str2})$ in particular implies
\begin{equation}\label{eq:F-str3}
\lambda^* \sum_{i=1}^n (\eta^i)^2 \le F^2(x,\eta) \le \frac1{\lambda} \sum_{i=1}^n (\eta^i)^2
\end{equation}
and thus the existence of positive constants $\kappa$ and $\kappa^*$ with
\begin{equation}\label{eq:F-str4}
{\kappa^*} F^2(x,\eta) \le \sum_{i,j=1}^n g_{ij}(x,\xi) \eta^i \eta^j \le \frac1\kappa F^2(x,\eta)
\end{equation}
for almost all $x \in U$ and all $\xi,\eta \in T_xM \setminus \{0\}$.
This (coordinate-free) inequality in turn implies
\begin{align}
F^2\bigg( x,\frac{\xi+\eta}{2} \bigg)
&\ge \frac{1}{2}F^2(x,\xi) +\frac{1}{2}F^2(x,\eta) -\frac1{4\kappa}F^2(x,\eta-\xi), \nonumber \\
F^2\bigg( x,\frac{\xi+\eta}{2} \bigg)
&\le \frac{1}{2}F^2(x,\xi) +\frac{1}{2}F^2(x,\eta) -\frac{\kappa^*}4 F^2(x,\eta-\xi) \label{eq:2uni}
\end{align}
for all $x\in U$ and $\xi,\eta \in T_xM$ (see \cite{BCL}, \cite{Ouni}).
For any subset $\Omega\subset M$, the largest constants $\kappa,\kappa^* \in (0,1]$ such that (\ref{eq:F-str4}) holds
for all $x\in\Omega$ will be denoted by $\kappa_\Omega$ and $\kappa^*_\Omega$.
The constants $1/\sqrt{\kappa^*_\Omega}$ and $1/\sqrt{\kappa_\Omega}$ are also known as {\it $2$-uniform convexity}
and {\it smoothness constants}.
Let us remark that $\kappa_\Omega=1$ (or $\kappa^*_\Omega=1$) if and only if $F(x,\cdot)$ is a Hilbert norm for each $x\in \Omega$.

\medskip

A nonnegative function $\| \cdot \|$ on $\R^n$ is called  {\it Minkowski norm} ---
and the pair $(\R^n, \| \cdot \|)$ is then called {\it Minkowski space} --- if $\|x\| >0$, $\|cx\| =c\|x\|$
and $\|x+y\| \le \|x\| +\|y\|$ hold for all $x,y \in \R^n \setminus \{0\}$ and $c>0$.
Thus a Finsler structure $F$ on $M$ induces for a.e. $x\in M$ a Minkowski norm $F(x,\cdot)$ on the tangent space $T_xM$.

Observe that there is a one-to-one correspondence between Minkowski norms
$\| \cdot \|: \R^n \rightarrow \R_+$ and convex, bounded open sets $B \subset \R^n$ containing the origin:
given $\| \cdot \|$, $B$ will be the open unit ball $\{ x \in \R^n \,:\, \|x\| <1 \}$;
given $B$, the associated Minkowski norm is defined by $\|x\| =\inf\{ c>0 \,:\, c^{-1}x \in B \}$.
Obviously, $\| \cdot \|$ will even be a norm if and only if $B$ is symmetric (i.e., $x \in B$ if and only if $-x \in B$).

The {\it reverse} Finsler structure $\overleftarrow{F}$ of $F$ is defined by $\overleftarrow{F}(x,\xi):=F(x,-\xi)$.
We say that $F$ is  {\it reversible} (or {\it absolutely homogeneous}) if $\overleftarrow{F}=F$.

\medskip

Usually in differential geometry, Finsler manifolds are assumed to be {\it smooth} in the sense that
$F$ is smooth on $TM\setminus\{0\}$. Note that we never require smoothness at the zero section.
Requiring that $F^2(x,\cdot)$ is $\mathcal{C}^2$ on all of $T_xM$ implies that $F(x,\cdot)$ is a Hilbert norm
(see \cite[Proposition 2.2]{Shvol}).

\subsection{The Legendre Transform}

For a Finsler structure $F$ on $M$, we define the dual structure $F^*: T^*M \rightarrow \R_+$ by
\[ F^*(x,\alpha) =\sup\{ \alpha \xi \,:\, \xi \in T_xM,\, F(x,\xi) \leq 1 \} \]
for $(x,\alpha) \in T^*M$ with regular $x$.
For each regular $x \in M$, we remark that $F^*(x,\cdot)$ is a Minkowski norm on $T_x^*M$ and set
\begin{equation}\label{eq:g*ij}
g^*_{ij}(x,\alpha) :=\frac{\partial^2}{\partial \alpha^i \partial \alpha^j} \bigg( \frac{1}{2}F^{*2}(x,\alpha) \bigg)
\end{equation}
for $\alpha=\sum_{i=1}^n \alpha^i dx^i \in T^*_xM \setminus \{0\}$ in a local coordinate system $(x^i)_{i=1}^n$.
(Here $F^{*2}(x,\cdot)$ is indeed twice differentiable on $T_x^*M \setminus \{0\}$,
see Lemma \ref{lm:g}(iii) below.)

The \emph{Legendre transform} or \emph{transfer map}  $J^*:T^*M \rightarrow TM$ assigns to each $\alpha \in T_x^*M$ with regular $x$
the unique maximizer of the function
\begin{equation}\label{eq:Leg}
\xi \ \mapsto \ \alpha \xi -\frac{1}{2}F^2(x,\xi) -\frac{1}{2}F^{*2}(x,\alpha)
\end{equation}
on $T_xM$.
(The last term is unnecessary but inserted for the sake of symmetry.)
The uniqueness is guaranteed by the strict convexity of $F(x,\cdot)$.
The vector $J^*(x,\alpha)$ can be characterized as the unique vector $\xi \in T_xM$
with $F(x,\xi)=F^*(x,\alpha)$ and $\alpha \xi =F^*(x,\alpha) F(x,\xi)$.
We can define $J:TM \rightarrow T^*M$ in an analogous way and then $J(x,\xi)$
is the unique maximizer of $(\ref{eq:Leg})$ as a function of $\alpha$.

We recall several standard properties of the Legendre transform which can be found in \cite[\S 14.8]{BCS} for instance.

\begin{lemma}\label{lm:g}
Fix regular $x \in M$.
\begin{itemize}
\item[{\rm (i)}] It holds that $J^*=J^{-1}$ on $T_x^*M$.

\item[{\rm (ii)}] For any $\alpha=\sum_{i=1}^n \alpha^i dx^i \in T_x^*M$
and $\xi=\sum_{i=1}^n \xi^i (\partial/\partial x^i) \in T_xM$, we have
\[ J^*(x,\alpha)=\sum_{i=1}^n \frac{\partial}{\partial \alpha^i} \bigg( \frac{1}{2}F^{*2}(x,\alpha) \bigg)
 \frac{\partial}{\partial x^i}, \quad
 J(x,\xi)=\sum_{i=1}^n \frac{\partial}{\partial \xi^i} \bigg( \frac{1}{2}F^2(x,\xi) \bigg) dx^i. \]

\item[{\rm (iii)}] $g^*_{ij}(x,\alpha)$ in $(\ref{eq:g*ij})$ is well-defined
for all $\alpha \in T^*_xM \setminus \{0\}$ and we have,
for all $\xi \in T_xM \setminus \{0\}$ and $\alpha \in T_x^*M \setminus \{0\}$,
\[ J(x,\xi) =g(x,\xi) \xi=\sum_{i,j=1}^n g_{ij}(x,\xi)\xi^i dx^j \in T_x^*M,\]
\[ J^*(x,\alpha) =g^*(x,\alpha)\alpha=\sum_{i,j=1}^n g^*_{ij}(x,\alpha) \alpha^i \frac{\partial}{\partial x^j} \in T_xM. \]
In particular, $g^*_{ij}(x,\alpha)$ is the inverse matrix of $g_{ij}(x,J^*(x,\alpha))$.

\item[{\rm (iv)}] For all $\xi,\eta \in T_xM$, we have
\begin{equation}
\label{eq:2uni-mod}
 \big( J(x,\eta)-J(x,\xi) \big)(\eta-\xi)\ge \kappa^* F^{2}(x,\eta-\xi).
\end{equation}

\item[{\rm (v)}] The dual structure $F^*$ satisfies estimates analogous to $(\ref{eq:F-str2})$, $(\ref{eq:F-str3})$,
$(\ref{eq:F-str4})$, $(\ref{eq:2uni})$ and $(\ref{eq:2uni-mod})$
with $\lambda^*$ and $\kappa^*$ in the place of $\lambda$ and $\kappa$, respectively, and vice versa.
\end{itemize}
\end{lemma}

\begin{proof}
The existence of $g^*_{ij}(x,\alpha)$ in (iii) is merely a consequence of the inverse function theorem (for $J$ and $J^*$).
As for (iv), since $g=\partial_\xi J$ by (ii), the mean value theorem implies that
$(J(x,\eta)-J(x,\xi))(\eta-\xi)=(\eta-\xi)g(x,\gamma)(\eta-\xi)$ for some $\gamma$ on the segment between $\xi$ and $\eta$.
Using $(\ref{eq:F-str4})$ the RHS can thus be estimated from below by $\kappa^* F^{2}(x,\eta-\xi)$.
\end{proof}

Note that at the origin $J^*(x,\cdot)$ is continuous but not differentiable (even if $F$ is smooth on $TM\setminus\{0\}$).

\begin{remark} \rm
Fixing a coordinate system, we may identify both $T_xM$ and $T_x^*M$ with the Euclidean space $\R^n$.
Given a vector $\xi\in T_xM$ of length $F(x,\xi)=1$, the vector $J^*(x,\xi)$ corresponds to the unit normal vector
at the point $\xi$ at the unit sphere in $T_xM$ (Figure 1).




For each regular $x \in M$ and $\xi \in T_xM \setminus \{0\}$, the map
\begin{equation}\label{eq:gxi}
F^\xi(x,\cdot): \ \eta \ \mapsto \ \left(\sum_{i,j=1}^n g_{ij}(x,\xi) \eta^i \eta^j\right)^{1/2}
\end{equation}
defines a Hilbert norm on $T_xM$.
It can be regarded as the best Hilbert norm approximation of the norm $F(x,\cdot)$ in directions close to $\xi$.
More precisely, if $\xi \in \partial B$ is a unit tangent vector, then the unit sphere $\partial B^{\xi}$
associated with the norm $F^{\xi}(x,\cdot)$ is the centered ellipse in $\R^n$ approximating
$\partial B$ up to second order at the point $\xi$ (Figure 2).




\end{remark}

\begin{example}\label{ex:Fstr} \rm
(i) Riemannian spaces:
Let $M=\R^n$ and $F$ be given by
\[ F^2(x,\xi)=\sum_{i=1}^n \xi^i a_{ij}(x)\xi^j \]
with a symmetric, positive-definite matrix $a(x)=(a_{ij}(x))$ on $\R^n$.
Then $g(x,\xi)=a(x)$ independently of $\xi$ and $J(x,\xi)=a(x)\xi$.
Moreover, $g^*(x,\alpha)=a(x)^{-1}$ and $J^*(x,\alpha)=a(x)^{-1} \alpha$.

(ii) $l^p$-spaces:
Let $M=\R^n$ and $F(x,\xi)=\| \xi \|_p=(\sum_{i=1}^n |\xi^i|^p)^{1/p}$ for some $1<p<\infty$.
Then
\[ J(x,\xi) =\| \xi \|_p^{2-p} \cdot (|\xi^i|^{p-2} \xi^i) \]
and $F^*(x,\alpha) =\| \alpha \|_{p^*}$ for the dual exponent $p^*$ satisfying $1/p +1/p^* =1$.
However, $\| \cdot \|_p$ is only $2$-uniformly convex if $1<p<2$ ($\kappa^*=p-1$, $\kappa=0$ in $(\ref{eq:2uni})$)
and only $2$-uniformly smooth if $2<p<\infty$ ($\kappa^*=0$, $\kappa=1/(p-1)$ in $(\ref{eq:2uni})$).
Therefore $\| \cdot \|_p$ is uniformly elliptic only when $p=2$.
Nevertheless, we can still consider the Laplacian (see the next chapter).

(iii) Deformation of Minkowski spaces:
Let $M=\R^n$ and $F(x,\xi) =\| \sigma(x)\xi \|$ for some invertible matrix $\sigma(x)$
and some Minkowski norm $\| \cdot \|$ on $\R^n$ which is strictly convex and twice differentiable on $\R^n \setminus \{0\}$.
(Case (i) is the particular case with Euclidean norm and $a(x)= \sigma(x)^T \sigma(x)$.)
Then we have $J(x,\xi)=\sigma(x)^T J_0(\sigma(x)\xi)$, $g(x,\xi)=\sigma(x)^T g_0(\sigma(x)\xi)\sigma(x)$
and $F^*(x,\alpha)=F^*_0(\alpha \sigma^{(-1)}(x))$,
where $J_0$, $g_0$ and $F^*_0$ are taken with respect to the original norm $\| \cdot \|$.

(iv) Hilbert geometry:
Let $D \subset \R^n$ be a bounded open convex domain with smooth boundary $\partial D$
such that $D \cup \partial D$ is strictly convex.
Given distinct $x,y \in D$, let $x' \in \partial D$ be the intersection of the half line $x+\R_+(x-y)$
with $\partial D$.
Similarly, let us denote by $y' \in \partial D$ the intersection of $y+\R_+(y-x)$ with $\partial D$.
Then the {\it Hilbert metric} $d_H$ is defined by
\[ d_H(x,y):=\log \bigg( \frac{|x'-y| \cdot |x-y'|}{|x'-x| \cdot |y-y'|} \bigg), \]
where $|\cdot|$ is the standard Euclidean norm.
If $D$ is the unit ball, then $(D,d_H)$ coincides with the Klein model of the hyperbolic space.
In general, $d_H$ arises from a Finsler metric of constant negative flag curvature (see \cite{Eg}).

(v) Teichm\"uller metric:
The Teichm\"uller metric on Teichm\"uller space is arguably the most famous Finsler metric in differential geometry.
It is known to be complete, while the Weil-Petersson metric is Riemannian and incomplete (see \cite{EE}, \cite{Wo}).
\end{example}

\subsection{Regularization}

Various of the results presented in this paper also will be true for more general Finsler structures,
not satisfying our basic regularity assumption (\ref{eq:F-str2}) with positive constants
$\lambda$ and $\lambda^*$ but just with nonnegative constants.
However, each Finsler structure $F$ of this type can easily be approximated by Finsler structures $F_{[\epsilon]}$
satisfying our assumptions.
We will illustrate this in the particular case of Minkowski norms on $\mathds{R}^n$.

For the sequel, we fix a Minkowski norm $\| \cdot \|$ and we denote by
$g(\xi)$ the Hessian of $\| \cdot \|^2/2$ at the point $\xi$.
We say that the Minkowski norm $\| \cdot \|$ is \emph{regular} if it satisfies $(\ref{eq:F-str2})$
with positive constants $\lambda$ and $\lambda^*$.

Moreover, we denote the Euclidean norm on $\mathds{R}^n$ by $| \cdot |$.
Note that the Hessian of $| \cdot |^2/2$ at each point is the identity matrix $\mathds1$.
We define the $\epsilon$-lower regularization of the Minkowski norm
$\| \cdot \|$ by $\|\xi\|_{\lfloor \epsilon\rfloor}=\sqrt{\|\xi\|^2+\epsilon|\xi|^2}$
and the $\epsilon$-upper regularization by
$\| \alpha \|_{\lceil \epsilon\rceil}=(\sqrt{ \| \alpha \|^{*2}+\epsilon| \alpha |^2})^*$.
Here $\| \cdot \|^*$ denotes the dual norm. Obviously, on the level of the
Hessians this means $g_{\lfloor \epsilon\rfloor}(\xi)=g(\xi)+\epsilon\mathds1$ and
$g^*_{\lceil \epsilon\rceil}(\alpha)=g^*(\alpha)+\epsilon\mathds1$
where of course $g^*(\alpha)$ is the Hessian of $\| \cdot \|^{*2}/2$
at the point $\alpha$. Recall also that $g^*(\alpha)$ is the inverse
of $g(J^*(\alpha))$. Moreover, we define the
$\epsilon$-regularization $g_{[\epsilon]}(\xi)$ of the matrix $g(\xi)$ by
\[ g_{[\epsilon]}(\xi)=\big( g(\xi)+\epsilon\mathds1 \big) \circ \big( \mathds1+\epsilon g(\xi) \big)^{-1}. \]
If we define in a similar way $g^*_{[\epsilon]}(\alpha)$ then it will be inverse to
$g_{[\epsilon]}(J^*(\alpha))$.
Obviously, for each $\epsilon>0$
\[ g_{[\epsilon]}(\xi)\ge\epsilon\mathds1, \qquad \qquad
 g^*_{[\epsilon]}(\alpha)\ge\epsilon\mathds1 \]
(in the sense of quadratic forms) for all $\xi$ and $\alpha$.
Finally, let us put
\[ \|\xi\|_{[\epsilon]}=\sqrt{\xi \cdot
g_{[\epsilon]}(\xi)\cdot\xi},\qquad
\|\alpha\|^*_{[\epsilon]}=\sqrt{\alpha \cdot
g_{[\epsilon]}(\alpha)\cdot\alpha}. \]
Then $\| \cdot \|_{[\epsilon]}$ and $\| \cdot \|^*_{[\epsilon]}$ are regular
Minkowski norms, dual to each other.
As $\epsilon$ goes to zero, they approximate the original norm $\| \cdot \|$ and its dual, respectively.

\subsection{Gradient Vectors and Distance}

For a weakly differentiable function $u:M \rightarrow \R$, define its \emph{gradient vector} by
\begin{equation}\label{eq:grad}
\Nabla u(x):=J^*\big( x,Du(x) \big)
\end{equation}
for every regular $x \in M$, where the \emph{derivative} $Du(x) \in T_x^*M$ is well-defined.
In a local coordinate system, we have $Du(x)=\sum_{i=1}^n (\partial u/\partial x^i)(x) dx^i$ and
\begin{equation*}
\Nabla u(x)=\sum_{i,j=1}^n g^*_{ij}\big( x,Du(x) \big) \frac{\partial u}{\partial x^i}(x) \frac{\partial}{\partial x^j}.
\end{equation*}
We remark that the nonlinearity descends from the Legendre transform to the gradient vector,
namely $\Nabla(u+v) \neq \Nabla u +\Nabla v$ in general.
For the same reason,  at points $x$ with $\Nabla u(x)=0$ the gradient vector field $\Nabla u$ is in general not differentiable
-- even if $(M,F)$ and $u$ are smooth -- but only continuous.

We define the \emph{distance function} $d:M \times M \rightarrow \R_+$ by
\begin{equation*}\label{dist-sup}
d(x,y) :=\sup \big\{ u(y)-u(x) \,:\, u \in \mathcal{C}^1(M), \,
 F\big( z,\Nabla u(z) \big) \le 1 \mbox{ for all } z \in M \big\}.
\end{equation*}
If $F$ is $\mathcal{C}^2$ on $TM \setminus \{0\}$, then this is equivalent to
\begin{equation*}\label{dist-inf}
d(x,y) =\inf_\gamma \int_0^1 F\big( \gamma,\dot{\gamma}(t) \big) \,dt
 =\inf_\gamma \bigg( \int_0^1 F^2\big( \gamma,\dot{\gamma}(t) \big) \,dt \bigg)^{1/2},
\end{equation*}
where the infimum is taken over all differentiable curves $\gamma:[0,1] \rightarrow M$ with $\gamma(0)=x$ as well as $\gamma(1)=y$.
For fixed $y \in M$, the distance function $x \mapsto d(y,x)$ satisfies $F(x,\Nabla d(y,x)) = 1$ for almost every $x \in M$
(more precisely, for all $x\in M\setminus(\{y\}\cup\mathrm{Cut}(y))$
where $\mathrm{Cut}(y)$ being the cut locus of $y$, see Chapter \ref{sc:Ricci}).
Moreover, the distance function $d$ has the following properties of a metric:
\begin{itemize}
\item $d(x,y) \ge 0$ for all $x,y \in M$ and $d(x,y)=0$ if and only if $x=y$;
\item $d(x,z) \le d(x,y)+d(y,z)$ for all $x,y,z \in M$.
\end{itemize}
Note that in general $d$ will not be symmetric.
The function $\overleftarrow{d}(x,y):=d(y,x)$ will be the distance function for the reverse Finsler structure $\overleftarrow{F}$ of $F$.
Locally $d$ and $\overleftarrow d$ are comparable thanks to the uniform ellipticity $(\ref{eq:F-str2})$.
We define the {\it forward} and {\it backward open balls} as
\[ B^+(x,r):=\{ y \in M \,:\, d(x,y)<r \}, \quad B^-(x,r):=\{ y \in M \,:\, d(y,x)<r \} \]
for $x \in M$ and $r>0$.
Closed balls are defined similarly.

\begin{example} \rm
For  each Minkowski space $(M,F)=(\R^n,\| \cdot \|)$ we have $d(x,y)=\| y-x \|$.
\end{example}

We say that a Finsler manifold $(M,F)$ is {\it forward complete}
if every {\it forward Cauchy sequence} is convergent.
That is to say, if a sequence $\{ x_n \}_{n \in \N} \subset M$ satisfies
$\lim_{N \to \infty} \sup_{N \le n<m} d(x_n,x_m) =0$, then there exists a point $x \in M$ such that
$\lim_{n \to \infty}d(x_n,x) =0$.
By the Hopf-Rinow theorem (cf.\ \cite[Theorem 6.6.1]{BCS}), the forward completeness is equivalent to
that every bounded forward closed ball is compact.
We can similarly define the {\it backward completeness} which is nothing but the forward completeness of $\overleftarrow{F}$.
They are not equivalent because a forward Cauchy sequence may not be a backward Cauchy sequence.
Nonetheless, the convergence $\lim_{n \to \infty} d(x_n,x)=0$ is equivalent to $\lim_{n \to \infty}d(x,x_n)=0$.

Observe from the definition of the Legendre transform that $\Nabla u(x)$ points into the direction
in which  $u$ increases the most.
That is to say,
\[ F\big( x,\Nabla u(x) \big) =\limsup_{y \to x} \frac{u(y)-u(x)}{d(x,y)}. \]
If $u$ is $\mathcal{C}^1$ and if $F$ is $\mathcal{C}^2$ on $TM \setminus \{0\}$, then we have
\[ -\int_0^l F\big( \gamma,\Nabla (-u)(\gamma) \big) F(\gamma,\dot{\gamma}) \,dt
 \le u\big( \gamma(l) \big) -u\big( \gamma(0) \big)
 \le \int_0^l F\big( \gamma,\Nabla u(\gamma) \big) F(\gamma,\dot{\gamma}) \,dt \]
for any $\mathcal{C}^1$-curve $\gamma:[0,l] \rightarrow M$.
Note the difference between $\Nabla(-u)$ and $-\Nabla u$.

\section{Finsler Laplacian}\label{sc:Lap}

Besides the Finsler structure $F$ on $M$, throughout the paper, we fix a \emph{measure} $m$ on $M$.
We always assume that this is locally bounded from above and below in terms of the volume form,
i.e., each point $\overline{x}\in M$ has a neighborhood $U$ with a local coordinate system $(x^i)_{i=1}^n$ such that
\begin{equation}\label{measure}
m(dx) =e^{-V(x)} dx^1 \cdots dx^n
\end{equation}
for some bounded measurable function $V:U \rightarrow \R$.

Given a smooth vector field $\Psi:M \rightarrow TM$, we define its \emph{divergence}
$\mathrm{div}\, \Psi:M \rightarrow \R$ through the identity
\begin{equation}\label{eq:div}
\int_M u\, \mathrm{div}\, \Psi \,dm =-\int_M \Psi u \,dm =-\int_M Du \cdot \Psi \,dm
\end{equation}
for all $u \in \mathcal{C}_c^{\infty}(M)$, where $Du \cdot \Psi=Du(\Psi)$ at $x$
denotes the canonical pairing between $T^*_xM$ and $T_xM$.
If in local coordinates $(x^i)_{i=1}^n$ the measure $m$ and the vector field $\Psi$ are given as
$m(dx)=e^{-V(x)} dx^1\cdots dx^n$ and $\Psi(x)=\sum_{i=1}^n \Psi^i (x) (\partial/\partial x^i)$
with differentiable functions $V$ and $\Psi^i$, then we have
\[ \mathrm{div}\, \Psi(x) =\sum_{i=1}^n \bigg\{
 \frac{\partial \Psi^i}{\partial x^i}(x) -\Psi^i(x) \frac{\partial V}{\partial x^i}(x) \bigg\}. \]
The concept of $\text{div}\, \Psi$ extends in an obvious way to smooth vector fields defined on open subsets
as well as to vector fields which are only weakly differentiable.

\begin{defn} \rm
A {\it Finsler space} is a triple $(M,F,m)$ consisting of a smooth, finite dimensional manifold $M$,
a Finsler structure $F$ on $M$ and a measure $m$ on $M$ as above.
\end{defn}

Note that, in this setting, the gradient depends on $F$ and the divergence on $m$.
Both $F$ and $m$ can be chosen independently.
The reason why we consider an arbitrary measure rather than constructive ones
(such as the Busemann-Hausdorff and the Holmes-Thompson measures) will be explained in Chapter \ref{sc:Ricci}.

Given an open set $\Omega \subset M$, the \emph{energy functional}
$\mathcal{E}_{\Omega}:H_{\mathrm{loc}}^1(\Omega) \rightarrow [0,\infty]$ on $\Omega$
is defined by
\[ \mathcal{E}_{\Omega}(u):=\frac{1}{2} \int_{\Omega} F^{*2}\big( x,Du(x) \big) \,m(dx)
 =\frac{1}{2} \int_{\Omega} F^2\big( x,\Nabla u(x) \big) \,m(dx). \]
 We will suppress $\Omega$ if $\Omega =M$, i.e., $\mathcal{E}:=\mathcal{E}_M$.
Clearly $\mathcal{E}_{\Omega}^{1/2}$ is convex and positively homogeneous.
Note that this energy functional coincides with Cheeger's one \cite{Ch} in terms of upper gradients (see also \cite{Sha}).
In order to make full use of Ricci curvature assumptions, this seems more suitable than the energy functional in terms of averaged difference quotients
as studied for instance in  \cite{Staop} or \cite{KS}.
The averaged energy incorporates a linearization of the operator (or the semigroup).
However,  the `canonical' Laplacians and heat semigroups on Finsler manifolds are always nonlinear
-- except in the Riemannian case.
See also \cite{Orec} for related work.

\medskip

Recall that the classes $L_{\mathrm{loc}}^2(\Omega)$ and $H_{\mathrm{loc}}^1(\Omega)$
are defined solely in terms of the manifold structure of $M$ (i.e., independent of the choices of $F$ and $m$).
Let $H^1(\Omega):=\{ u \in H^1_{\mathrm{loc}}(\Omega) \,:\, \mathcal{E}_{\Omega}(u) <\infty \}$
and $H_0^1(\Omega)$ be the closure of $\mathcal{C}_c^{\infty}(\Omega)$ (or, equivalently, $H_c^1(\Omega)$)
in $H^1(\Omega)$ with respect to the (Minkowski) norm $\|u\|_{H^1} :=\|u\|_{L^2} +\mathcal{E}_{\Omega}(u)^{1/2}$.
The dual space to $H^1_0(\Omega)$ is denoted by $H^{-1}(\Omega)$.

Define the \emph{energy functional with Dirichlet boundary conditions} $\mathcal{E}^0_{\Omega}: L^2(\Omega) \rightarrow [0,\infty]$
by $\mathcal{E}^0_\Omega(u):=\mathcal{E}_\Omega(u)$ for $u\in  H_0^1(\Omega)$ and
$\mathcal{E}^0_\Omega(u):=+\infty$ else.
The \emph{ground state energy} (inverse Poincar\'e constant) is given by
\[ \chi_\Omega :=\inf \left\{ 2\mathcal{E}_{\Omega}(u) \,:\, u\in H^1_0(\Omega),\, \|u\|_{L^2}=1\right\}. \]
If $\chi_\Omega=0$ (e.g., if $\Omega$ is compact) then it is more convenient to consider
\begin{equation}\label{eq:nuM}
\overline\chi_\Omega :=\inf \left\{ 2\mathcal{E}_{\Omega}(u) \,:\,
 u\in H_0^1(\Omega),\, \|u\|_{L^2}=1,\, \int_\Omega u\,dm=0\right\}.
\end{equation}

\begin{lemma}\label{lm:E}
\begin{itemize}
\item[{\rm (i)}]  The energy functional $\mathcal{E}^0_\Omega$ is lower semicontinuous on $L^2(\Omega)$.

\item[{\rm (ii)}] If $\Omega$ is relatively compact, then  $H^1_0(\Omega)$  is proper in the sense that
every bounded sequence in $(H^1_0(\Omega),\| \cdot \|_{H^1})$ contains a convergent subsequence.

\item[{\rm (iii)}]
If $\Omega$ is relatively compact and connected with non-polar boundary, then $\chi_\Omega>0$.
If $M$ is compact, then $\overline\chi_M>0$.

\item[\rm (iv)] The functional $\mathcal{E}^0_\Omega$ is $K$-convex on $L^2(\Omega,m)$ with $K=\chi_\Omega \kappa_\Omega$.
Moreover, for each $C\in\R$ it is $\overline K$-convex on the convex set $\{u\in L^2(\Omega): \ \int_\Omega u\,dm=C\}$ with
$\overline K=\overline\chi_\Omega \kappa_\Omega$.
\end{itemize}
\end{lemma}

\begin{proof} (i) -- (iii) are standard facts.
In fact, we can reduce (ii) and (iii) to a Riemannian structure (bi-Lipschitz) equivalent to $F$.

(iv) Recall that the dual version of (\ref{eq:2uni}) states
\[ F^{*2}\bigg( x,\frac{\alpha+\beta}{2} \bigg)
 \le \frac{1}{2}F^{*2}(x,\alpha) +\frac{1}{2}F^{*2}(x,\beta) -\frac{\kappa_{\Omega}}4 F^{*2}(x,\beta-\alpha) \]
for all $x \in \Omega$ and $\alpha, \beta \in T^*_xM$.
Hence,
\[ \mathcal{E}^0_\Omega\left(\frac{u+v}2\right)
 \le\frac12\mathcal{E}^0_\Omega\left(u\right)+\frac12\mathcal{E}^0_\Omega\left(v\right)
 -\frac{\kappa_{\Omega}}{4}\mathcal{E}^0_\Omega\left({u-v}\right). \]
The last term can be estimated by
$2\mathcal{E}^0_\Omega({u-v})\ge \chi_\Omega\|u-v\|^2_{L^2}$ for all $u,v\in L^2$ and by
$2\mathcal{E}^0_\Omega({u-v})\ge \overline\chi_\Omega\|u-v\|^2_{L^2}$ if in addition $\int udm=\int vdm$.
This proves the $K$- (and $\overline K$-, respectively) convexity.
\end{proof}

\medskip

We define the \emph{Finsler Laplacian} $\bm{\Delta}$ acting on 
functions $u \in H^1_{\mathrm{loc}}(\Omega)$ formally by 
$\bm{\Delta}u :=\mathrm{div}\, (\Nabla u)$ (cf.\ \cite{Shlap}, 
\cite{BKJ}). To be more precise, $\bm{\Delta}u$ is the 
distributional Laplacian defined through the identity 
\[ \int_{\Omega} v \bm{\Delta}u \,dm =-\int_{\Omega} Dv(\Nabla u) \,dm \]
for all $v \in H^1_0(\Omega)$ (or, equivalently, for all $v \in \mathcal{C}^{\infty}_c(\Omega)$).
Recall that at points $x$ with $\Nabla u(x)=0$ the function $\Nabla u$ in general will be  not differentiable
(even if the function $u$ itself and the norm $F$ will be smooth).
Note the sign convention: our Laplacian is a negative operator, i.e.,
\[ \int_{\Omega} u \bm{\Delta}u \,dm \le 0 \]
for all $u \in H^1_0(\Omega)$ (and equality holds if and only if $u$ is constant a.e.\ on each connected component of $\Omega$).
The Finsler Laplacian is a linear operator if and only if $F$ is a Riemannian structure
(i.e., $F(x,\cdot)$ is a Hilbert norm for a.e.\ $x \in M$).

Given $g \in L_{\mathrm{loc}}^2(\Omega)$, a function $u \in 
H_{\mathrm{loc}}^1(\Omega)$ is called a \emph{weak solution} of 
$\bm{\Delta}u=g$ in $\Omega$ if 
\[ -\int_{\Omega} Dv(\Nabla u) \,dm =\int_{\Omega} vg \,dm \]
for all $v \in H_0^1(\Omega)$. A function $u \in 
H^1_{\mathrm{loc}}(\Omega)$ is said to be \emph{weakly harmonic} on 
$\Omega$ if it is a weak solution of $\bm{\Delta}u=0$ in $\Omega$. 

\begin{lemma}\label{lm:wharm}
A function $u \in H_{\mathrm{loc}}^1(\Omega)$ is weakly harmonic on $\Omega$ if and only if it is a minimizer
of the energy functional $\mathcal{E}_{\Omega'}$ on each open set $\Omega'$ relatively compact in $\Omega$, i.e.,
\[ \mathcal{E}_{\Omega'}(u) =\inf\{ \mathcal{E}_{\Omega'}(u+v) \,:\, v \in H_0^1(\Omega') \}. \]
A function $u \in H^1(\Omega)$ is weakly harmonic on $\Omega$ if and only if it is a minimizer of $\mathcal{E}_{\Omega}$, i.e.,
\[ \mathcal{E}_{\Omega}(u) =\inf\{ \mathcal{E}_{\Omega}(u+v) \,:\, v \in H_0^1(\Omega) \}. \]
\end{lemma}

\begin{proof}
The first claim immediately follows from the calculation
\begin{align*}
\frac{d}{d\delta}\Big|_{\delta=0} \mathcal{E}_{\Omega'}(u+\delta v)
&= \int_{\Omega'} \frac{d}{d\delta} \Big|_{\delta=0} \left(\frac{1}{2} F^{*2} \big( x,Du(x)+\delta Dv(x) \big) \right) \,m(dx) \\
&= \int_{\Omega'} Dv(\Nabla u) \,dm,
\end{align*}
where for the second equation we used Lemma \ref{lm:g}(ii). For the second claim, in addition we take the estimate
$\int_{\Omega} Dv(\Nabla u) \,dm \le 2 \mathcal{E}_{\Omega}(v)^{1/2} \mathcal{E}_{\Omega}(u)^{1/2}$
into account.
\end{proof}

We also introduce a weighted Laplacian associated with a Riemannian structure
induced from the gradient vector field of some function.
Given a function $u \in H^1_{\mathrm{loc}}(\Omega)$, we define the Riemannian tensor $g^{(u)}$ on $M$ by
\begin{equation}\label{g-u}
g^{(u)}(x):= g_{\Nabla u(x)}:=g\big( x,\Nabla u(x) \big) =\Big( g_{ij}\big( x,\Nabla u(x) \big) \Big)_{i,j}
\end{equation}
for each $x\in M$ where $\Nabla u(x) \in T_xM$ is well-defined and nonzero.
Otherwise, we put $g^{(u)}(x):=g(x,Z(x))$ for some fixed nonvanishing vector field $Z$ on $M$.
Note that its inverse is given by
\[ g^{(u)}(x)^{-1} :=g^{*(u)}(x)=g^*\big( x,Du(x) \big). \]
For each $u \in H_{\mathrm{loc}}^1(\Omega)$, define the weighted Laplacian $\Delta^{(u)}$ acting on
functions $w \in H^1_{\mathrm{loc}}(\Omega)$ in the sense of distributions by $\Delta^{(u)}w:=\mathrm{div}\, ({g^{*(u)}}Dw)$.
Here
\[ g^{*(u)}(x) Dw(x) =\sum_{i,j=1}^n g^*_{ij}\big( x,Du(x) \big)
 \frac{\partial w}{\partial x^i}(x) \frac{\partial}{\partial x^j} \in T_xM. \]

\begin{lemma}\label{lm:Lap}
For any $u \in H^1_{\mathrm{loc}}(\Omega)$, we have 
$\bm{\Delta}u=\Delta^{(u)}u $ in the sense of distributions on 
$\Omega$. More precisely, for all $v \in H_0^1(\Omega)$, it holds 
that 
\[ \int_{\Omega} v \Delta^{(u)}u \,dm =\int_{\Omega} v \bm{\Delta} u \,dm. \]
\end{lemma}

\begin{proof}
Lemma \ref{lm:g}(iii), $(\ref{eq:grad})$ and the definition of the divergence $(\ref{eq:div})$ yield that
\[ \int_{\Omega} v \Delta^{(u)}u \,dm =-\int_{\Omega} Dv(g^{*(u)}Du) \,dm
 =-\int_{\Omega} Dv(\Nabla u) \,dm =\int_{\Omega} v \bm{\Delta}u \,dm. \]
\end{proof}

We will use the above lemma to show a generalized Laplacian comparison theorem in Theorem \ref{th:Lcom}
as well as Corollary \ref{cr:Lcom}.
Compare them with the following remark.

\begin{remark} \rm
Let $(\R^n,\| \cdot \|,m)$ be a Minkowski space equipped with the Lebesgue measure, and put $u(x)=f(\|x-y\|)$
for some nondecreasing $\mathcal{C}^2$-function $f$ on $\R_+$ and some fixed point $y \in \R^n$.
Then we have, for any $x \neq y$,
\begin{equation}\label{eq:Ld1}
\bm{\Delta}u(x) =f''(\|x-y\|) +\frac{n-1}{\|x-y\|} f'(\|x-y\|).
\end{equation}
In particular, $\bm{\Delta}(\| x-y \|^2) =2n$. If $f$ is 
nonincreasing, then an analogous result holds true for 
$v(x)=f(\|y-x\|)$, namely 
\begin{equation}\label{eq:Ld2}
\bm{\Delta}v(x) =f''(\|y-x\|) +\frac{n-1}{\|y-x\|} f'(\|y-x\|).
\end{equation}
This is because, for nonincreasing $f$, the right-hand side of 
$(\ref{eq:Ld1})$ coincides with $-\bm{\Delta}(-u)(x) 
=\overleftarrow{\bm{\Delta}}u(x)$, where 
$\overleftarrow{\bm{\Delta}}$ stands for the Finsler Laplacian for 
the reverse Finsler structure $\overleftarrow{F}(x,\xi)=\|-\xi\|$. 
Similarly, for nondecreasing $f$, the right-hand side of 
$(\ref{eq:Ld2})$ coincides with $\overleftarrow{\bm{\Delta}}v(x)$. 
\end{remark}

\begin{proof}
We deduce from Lemma \ref{lm:g}(ii) that
\[ Du(x) =f'(\|x-y\|) D(\|x-y\|)=\frac{f'(\|x-y\|)}{\|x-y\|} J(x-y). \]
For nondecreasing $f$, we find
\[ \bm{\Delta}u(x) =\mathrm{div}\, J^* \bigg( \frac{f'(\|x-y\|)}{\|x-y\|} J(x-y) \bigg)
 =\mathrm{{div}} \bigg( \frac{f'(\|x-y\|)}{\|x-y\|}(x-y) \bigg) \]
which implies the claim since $\sum_{i=1}^n (\partial \|x-y\|/\partial x^i)(x^i-y^i) =\|x-y\|$
by Euler's theorem (cf.\ \cite[Theorem 1.2.1]{BCS}).
\end{proof}

\section{The Heat Equation -- Global Solutions}\label{sc:global}

To simplify the presentation, we will assume throughout this chapter
that the general assumptions of the previous chapters are satisfied.
That is, $(M,F,m)$ is a Finsler space with a Finsler structure $F$ satisfying (\ref{eq:F-str2}) and a measure $m$ satisfying (2.1).
We remark, however, that instead of $(\ref{eq:F-str2})$ for the sequel it suffices to assume that
$F(x,\cdot)$ is strictly convex and differentiable on $T_xM \setminus \{0\}$ for a.e.\ $x$.

\begin{defn} \rm
We say that $u$ is a {\it global solution} to the heat equation 
$\partial_t u =\bm{\Delta}u$ on $[0,T] \times M$ if $u\in 
L^2([0,T],H^1_0(M)) \cap H^1([0,T],H^{-1}(M))$ and if, for every $t 
\in [0,T]$ and $v \in H^1_0(M)$, it holds that 
\begin{equation}\label{global-heat}
\int_M v \partial_t u_t \,dm =-\int_M Dv(\Nabla u_t) \,dm.
\end{equation}
Here $u_t(x):=u(t,x)$.
To be more precise, our global solutions are always global solutions \emph{with Dirichlet boundary conditions}.
\end{defn}

\begin{remark} \rm
(i) The condition $u \in L^2([0,T],H^1_0(M)) \cap H^1([0,T],H^{-1}(M))$ implies that
$u \in \mathcal{C}([0,T],L^2(M))$ (see, e.g., \cite[p.287]{Ev}).
Equivalently we could require that the above identity $(\ref{global-heat})$ holds
for all $v \in  L^2([0,T],H^1_0(M))$ and a.e.\ $t \in [0,T]$.

(ii)
If $M$ is compact, then every global solution $u$ to the heat equation is \emph{mass preserving},
i.e., $\int_M u_t \,dm=\int_M u_0 \,dm$ holds for all $t$.
Indeed, choosing constant $v \equiv 1 \in L^2([0,T],H^1_0(M))$ as test function yields the claim.
\end{remark}

Now we are going to construct a global solution to the heat equation
as gradient flow of the energy functional $\mathcal{E}^0$ on $L^2(M)$.
Since $\mathcal{E}^0$ is a convex function on the Hilbert space $L^2(M)$, we can apply Crandall and Liggett's classical
technique \cite{CL} (see also \cite{Ma}, \cite{AGS} for generalizations to curved spaces).
To simplify notation, we use $\mathcal{E}$ instead of $\mathcal{E}^0$ but take care to evaluate it only on $H^1_0(M)$.

Given $u \in H^1_0(M)$, we define
\[ |\nabla(-\mathcal{E})|(u)
 :=\max\bigg\{ 0,\limsup_{v \to u}\frac{\mathcal{E}(u)-\mathcal{E}(v)}{\| v-u \|_{L^2}} \bigg\}, \]
where $v \in H^1_0(M)$ and the convergence $v \to u$ is with respect to the $L^2$-norm.
Note that the convexity of $\mathcal{E}$ implies that $|\nabla(-\mathcal{E})|(u)=0$ holds
if and only if $u$ is a minimizer of $\mathcal{E}$ on $H^1_0(M)$.

\begin{lemma}\label{lm:gvec}
If $0< |\nabla(-\mathcal{E})|(u) <\infty$, then there exists unique $v \in L^2(M)$ satisfying
$\| v \|_{L^2}=|\nabla(-\mathcal{E})|(u)$ as well as
\[ \lim_{t \downarrow 0} \frac{\mathcal{E}(u) -\mathcal{E}(u+tv)}{t\| v \|_{L^2}} =|\nabla(-\mathcal{E})|(u). \]
\end{lemma}

\begin{proof}
Take a sequence $\{ \hat{v}_i \}_{i \in \N} \subset H^1_0(M) \setminus \{0\}$ such that
\[ \lim_{i \to \infty}\frac{\mathcal{E}(u)-\mathcal{E}(u+\hat{v}_i)}{\| \hat{v}_i \|_{L^2}}
 =|\nabla(-\mathcal{E})|(u). \]
We put $v_i:=(|\nabla(-\mathcal{E})|(u)/\| \hat{v}_i \|_{L^2}) \cdot \hat{v}_i$
and deduce from the convexity of $\mathcal{E}$ that
\[ \lim_{i \to \infty} \lim_{t \downarrow 0} \frac{\mathcal{E}(u) -\mathcal{E}(u+tv_i)}{t\| v_i \|_{L^2}}
 \ge \lim_{i \to \infty} \frac{\mathcal{E}(u) -\mathcal{E}(u+\hat{v}_i)}{\| \hat{v}_i \|_{L^2}}
 =|\nabla(-\mathcal{E})|(u). \]
Thus we have
$\lim_{i \to \infty} \lim_{t \downarrow 0} \{ \mathcal{E}(u) -\mathcal{E}(u+tv_i) \}/t\| v_i \|_{L^2}=|\nabla(-\mathcal{E})|(u)$.
Moreover, for any $i,j \ge 1$, we see
\begin{align*}
|\nabla(-\mathcal{E})|(u)
&\ge \lim_{t \downarrow 0} \frac{\mathcal{E}(u) -\mathcal{E}(u+t(v_i+v_j)/2)}{t\| (v_i+v_j)/2 \|_{L^2}} \\
&\ge \frac{|\nabla(-\mathcal{E})|(u)}{\| (v_i+v_j)/2 \|_{L^2}} \lim_{t \downarrow 0} \bigg\{
 \frac{\mathcal{E}(u) -\mathcal{E}(u+tv_i)}{2t\| v_i \|_{L^2}}
 +\frac{\mathcal{E}(u) -\mathcal{E}(u+tv_j)}{2t\| v_j \|_{L^2}} \bigg\}.
\end{align*}
This implies that $\lim_{i,j \to \infty} \| v_i+v_j \|_{L^2} =2|\nabla(-\mathcal{E})|(u)$.
Hence $\{ v_i \}_{i \in \N}$ is a Cauchy sequence and converges to some $v \in L^2(M)$.
Uniqueness is deduced in a similar way.
\end{proof}

We define $\nabla(-\mathcal{E})(u):=v$ using $v \in L^2(M)$ as in Lemma \ref{lm:gvec} above
and call $\nabla(-\mathcal{E})(u)$ the {\it gradient vector} of $-\mathcal{E}$ at $u$.
We simply set $\nabla(-\mathcal{E})(u):=0 \in L^2(M)$ if $|\nabla(-\mathcal{E})|(u)=0$.

For $u_0 \in H^1_0(M)$ and $\delta>0$, we denote by $U_{\delta}(u_0) \in H^1_0(M)$ the unique minimizer of the function
\begin{equation}\label{eq:disc}
u\ \mapsto\ \mathcal{E}(u)+\frac{\| u-u_0 \|_{L^2}^2}{2\delta}.
\end{equation}
This can be regarded as a discrete approximation of a gradient flow of $\mathcal{E}$.
In fact, $(U_{t/n})^n(u_0)$ converges to a continuous curve $u:\R_+ \rightarrow H^1_0(M)$ with $u(0)=u_0$
as $n$ goes to infinity, and $u_t:=u(t)$ satisfies the following properties
(see, e.g., \cite[Theorem 1.13 \& Section 2]{Ma}):
\begin{itemize}
\item[(i)]
The curve $t \mapsto u_t$ in $L^2(M)$ is locally Lipschitz continuous on $(0,\infty)$ and satisfies,
for a.e.\ $t \ge 0$,
\begin{equation}\label{eq:u/d}
\lim_{\delta \to 0} \frac{\| u_{t+\delta}-u_t \|_{L^2}}{|\delta|} =|\nabla(-\mathcal{E})|(u_t).
\end{equation}
In particular, we have $|\nabla(-\mathcal{E})|(u_t)<\infty$ at every $t>0$.

\item[(ii)]
For a.e.\ $t \ge 0$, it holds that
\begin{equation}\label{eq:E/d}
\lim_{\delta \to 0} \frac{\mathcal{E}(u_t) -\mathcal{E}(u_{t+\delta})}{\delta} =|\nabla(-\mathcal{E})|(u_t)^2.
\end{equation}
\end{itemize}

Thanks to (i) and (ii) above, a similar discussion to the proof of Lemma \ref{lm:gvec} ensures that, for a.e.\ $t>0$,
\begin{equation}\label{eq:diff}
\lim_{\delta \to 0} \bigg\| \frac{u_{t+\delta} -u_t}{\delta} -\nabla(-\mathcal{E})(u_t) \bigg\|_{L^2} =0.
\end{equation}
In other words, $\partial_t u_t=\nabla(-\mathcal{E})(u_t)$ in the weak sense.
If we replace the limit with the right limit $\lim_{\delta \downarrow 0}$, then equations $(\ref{eq:u/d})$ and $(\ref{eq:E/d})$
hold for all $t \ge 0$ and $(\ref{eq:diff})$ holds for all $t>0$.
In addition, we find
\begin{equation}\label{eq:asym}
\lim_{\delta \downarrow 0} \frac{\| u_{t+\delta}-U_{\delta}(u_t) \|_{L^2}}{\delta}=0
\end{equation}
for all $t>0$ along the same lines as \cite[Lemma 6.4]{Ogra}.

\begin{thm}\label{th:glob}
For each $u_0 \in H^1_0(M)$ and $T>0$, there exists a global 
solution $u$ to the heat equation which lies in $L^2([0,T],H^1_0(M)) 
\cap H^1([0,T],L^2(M))$. Moreover, for each $t \in (0,T)$, the 
distributional Laplacian $\bm{\Delta}u_t$ is absolutely continuous 
with respect to $m$ and its density function is 
$\nabla(-\mathcal{E})(u_t)$. In particular, $\partial_t u_t 
=\bm{\Delta}u_t$ in the weak sense $($see $(\ref{eq:diff}))$ and we 
have 
\begin{equation}\label{e-decay}
\lim_{\delta \downarrow 0} \frac{\mathcal{E}(u_t) -\mathcal{E}(u_{t+\delta})}{\delta}
 =|\nabla(-\mathcal{E})|(u_t)^2 =\| \nabla(-\mathcal{E})(u_t) \|_{L^2}^2
 =\| \bm{\Delta}u_t \|_{L^2}^2
\end{equation}
for all $t>0$.
\end{thm}

\begin{proof}
Let $u:\R_+ \rightarrow H^1_0(M)$ be the gradient curve of $\mathcal{E}$ constructed as the limit curve
of the discrete approximation $(\ref{eq:disc})$.
Note that $\int_0^T \mathcal{E}(u_t)^2 \,dt \le T\mathcal{E}(u_0)^2 < \infty$ and
\[ \int_0^T \| \partial_t u_t \|_{L^2}^2 \,dt =\int_0^T |\nabla(-\mathcal{E})|(u_t)^2 \,dt
 =\mathcal{E}(u_0) -\mathcal{E}(u_T) <\infty. \]
Thus we observe $u \in L^2([0,T],H^1_0(M)) \cap H^1([0,T],L^2(M))$.

We shall show that, for any $v \in L^2([0,T],H^1_0(M)) \cap H^1([0,T],L^2(M))$ and $0 \le t_0<t_1 \le T$,
it holds that
\begin{equation}\label{eq:w-heat}
\int_M v_{t_1}u_{t_1} \,dm -\int_M v_{t_0}u_{t_0} \,dm
 =\int_{t_0}^{t_1} \int_M \bigg\{ \frac{\partial v_t}{\partial t} u_t -Dv_t(\Nabla u_t) \bigg\} \,dm \,dt.
\end{equation}

Fix $t \in (t_0,t_1)$.
Given small $\delta, \varepsilon>0$, consider unique $w^t_{\delta}:=U_{\delta}(u_t)$ minimizing the function
$w \mapsto \mathcal{E}(w)+\|w-u_t\|_{L^2}^2/2\delta$
and put $\tilde{w}^t_{\delta,\varepsilon}:=w^t_{\delta} +\varepsilon v_t$.
Then the choice of $w^t_{\delta}$ yields
\[ \mathcal{E}(\tilde{w}^t_{\delta,\varepsilon}) +\frac{\| \tilde{w}^t_{\delta,\varepsilon} -u_t \|_{L^2}^2}{2\delta}
 -\mathcal{E}(w^t_{\delta}) -\frac{\| w^t_{\delta} -u_t \|_{L^2}^2}{2\delta} \ge 0. \]
Firstly, we have
\[ \| \tilde{w}^t_{\delta,\varepsilon} -u_t \|_{L^2}^2 -\| w^t_{\delta} -u_t \|_{L^2}^2
 =2\varepsilon \langle w^t_{\delta}-u_t,v_t \rangle_{L^2} +\varepsilon^2 \|v_t\|_{L^2}^2, \]
and hence
\begin{align*}
\lim_{\varepsilon \to 0} \frac{\| \tilde{w}^t_{\delta,\varepsilon} -u_t \|_{L^2}^2 -\| w^t_{\delta} -u_t \|_{L^2}^2}{\varepsilon}
&=2\langle w^t_{\delta}-u_t,v_t \rangle_{L^2} \\
&=2\bigg\{ \int_M v_t w^t_{\delta} \,dm -\int_M v_tu_t \,dm \bigg\}.
\end{align*}

Secondly, it follows from Lemma \ref{lm:g}(ii) that
\begin{align*}
\lim_{\varepsilon \to 0} \frac{\mathcal{E}(\tilde{w}^t_{\delta,\varepsilon})-\mathcal{E}(w^t_{\delta})}{\varepsilon}
&=\lim_{\varepsilon \to 0} \frac{1}{2\varepsilon}
 \int_M \big\{ F^{*2}(Dw^t_{\delta} +\varepsilon Dv_t) -F^{*2}(Dw^t_{\delta}) \big\} \,dm \\
&=\int_M Dv_t(\Nabla w^t_{\delta}) \,dm.
\end{align*}
Therefore
\begin{align*}
\liminf_{\delta \downarrow 0} \frac{1}{\delta} \bigg\{ \int_M v_t w^t_{\delta} \,dm -\int_M v_tu_t \,dm \bigg\}
&\ge -\lim_{\delta \downarrow 0} \int_M Dv_t(\Nabla w^t_{\delta}) \,dm \\
&= -\int_M Dv_t(\Nabla u_t) \,dm.
\end{align*}
The second equality follows from the choice of $w^t_{\delta}$.
In fact, $\{ Dw^t_{\delta} \}_{\delta>0}$ is a Cauchy sequence converging to $Du_t$ as $\delta$ tends to zero.
Together with $(\ref{eq:asym})$, we have
\begin{align*}
&\liminf_{\delta \downarrow 0} \frac{1}{\delta} \bigg\{ \int_M v_{t+\delta}u_{t+\delta} \,dm -\int_M v_tu_t \,dm \bigg\} \\
&= \liminf_{\delta \downarrow 0} \frac{1}{\delta}\bigg\{ \int_M (v_{t+\delta}-v_t)u_{t+\delta} \,dm
 +\int_M v_tu_{t+\delta} \,dm -\int_M v_tu_t \,dm \bigg\} \\
&= \int_M \frac{\partial v_t}{\partial t}u_t \,dm
 +\liminf_{\delta \downarrow 0} \frac{1}{\delta} \bigg\{ \int_M v_t w^t_{\delta} \,dm -\int_M v_tu_t \,dm \bigg\} \\
&\ge \int_M \bigg\{ \frac{\partial v_t}{\partial t}u_t -Dv_t(\Nabla u_t) \bigg\} \,dm.
\end{align*}
We obtain the reverse inequality by exchanging $v$ with $-v$, and complete the proof of $(\ref{eq:w-heat})$.

By virtue of $(\ref{eq:diff})$, choosing time independent $v \in H^1_0(M)$ in $(\ref{eq:w-heat})$ shows that
\[ \int_M v \nabla(-\mathcal{E})(u_t) \,dm
 =\lim_{\delta \downarrow 0} \int_M v \frac{u_{t+\delta}-u_t}{\delta} \,dm
 =-\int_M Dv(\Nabla u_t) \,dm =\int_M v \bm{\Delta}u_t \,dm. \]
Hence $\bm{\Delta}u_t$ is absolutely continuous with respect to $m$
and the density function is nothing but $\nabla(-\mathcal{E})(u_t)$.
\end{proof}

The following proposition ensures that the gradient flow constructed as above is actually a unique solution
to the heat equation.
In particular, for each $u_0 \in L^2(M)$, it allows to construct a unique gradient curve $(u_t)_{t \ge 0}$ starting from $u_0$
as the limit of a sequence of gradient curves $(u^{(n)}_t)_{t \ge 0} \subset H^1_0(M)$
such that $u^{(n)}_0$ tends to $u_0$ in $L^2(M)$ as $n$ goes to infinity.
We denote this curve by $(P_tu_0)_{t \ge 0}$.
The map $P_t:u \mapsto P_t u$ defines a (non-expanding) semigroup of nonlinear operators on $L^2(M)$,
and we call it the {\it heat semigroup}.

\begin{prop}\label{pr:uniq}
For all global solutions $u,v$ to the heat equation, we have
\begin{align}
&\partial_t \left( \frac12 \|u_t\|^2_{L^2} \right) =-2\mathcal{E}(u_t), \label{l2-decay} \\
&\partial_t \left( \frac12 \|u_t-v_t\|^2_{L^2} \right) \le -2\kappa_M \mathcal{E}(u_t-v_t) \le 0 \label{l2-unique}
\end{align}
with $\kappa_M$ as introduced in $(\ref{eq:F-str4})$.
\end{prop}

\begin{proof}
Assuming that both $u$ and $v$ are global solutions to the heat equation and choosing $u-v$ as test function
(for each of these solutions) yields
\begin{align*}
\partial_t \left( \frac12\|u_t-v_t\|^2_{L^2} \right)
&= \int_M (u_t-v_t)( \partial_t u_t -\partial_t v_t) \,dm \\
&= -\int_M (Du_t-Dv_t)(\Nabla u_t -\Nabla v_t) \,dm.
\end{align*}
In the case $v \equiv 0$ the last term obviously coincides with $- 2\mathcal{E}(u_t)$ which proves the first claim.

In the general case, the last term of the previous identities can be estimated from above according to
Lemma \ref{lm:g}(v) which asserts
\[ (Du_t-Dv_t)(\Nabla u_t -\Nabla v_t) \ge \kappa_M F^{*2}(Du_t-Dv_t). \]
This proves the second claim.
\end{proof}

\begin{cor}\label{cr:cont}
For all global solutions $u,v$ to the heat equation, we have
\[ \|u_t-v_t\|_{L^2} \le e^{-\kappa_M \chi_M t} \cdot\|u_0-v_0\|_{L^2}. \]
If in addition $\int_M u_0 \,dm =\int_M v_0 \,dm$,
then in the estimate above $\chi_M$ can be replaced by $\overline\chi_M$, i.e.,
\[ \|u_t-v_t\|_{L^2} \le e^{-\kappa_M \overline\chi_M t}\cdot \|u_0-v_0\|_{L^2}; \]
if $v\equiv0$ then $\kappa_M$ can be replaced by $1$, i.e.,
\[ \|u_t\|_{L^2} \le e^{-\chi_M t} \cdot\|u_0\|_{L^2}. \]
\end{cor}

\begin{proof}
The estimates follow immediately from Proposition \ref{pr:uniq} (together with the definition of $\chi_M$ and $\overline\chi_M$)
and an application of Gronwall's lemma.
\end{proof}

The previous are the usual contraction properties of gradient flows,
for $\mathcal{E}^0$ is $\kappa_M \chi_M$-convex on the Hilbert space $L^2(M)$ (Lemma \ref{lm:E}(iv)).
Recall that compactness of $M$ will imply $\kappa_M>0$ and $\overline\chi_M>0$.
A slightly modified argument will yield contraction in $L^p(M)$ for each $p$.

\begin{thm}\label{thm:cont}
For all $p\in[1,\infty]$ and all global solutions $u,v$ to the heat equation, we have
\[ \|u_t-v_t\|_{L^p} \le \exp\left(-\frac{4(p-1)}{p^2}\kappa_M \chi_M t\right) \cdot \|u_0-v_0\|_{L^p}. \]
If $v\equiv0$ then in the estimate above $\kappa_M$ can be replaced by $1$.
\end{thm}

\begin{proof} Assume $1<p<\infty$.
(The cases $p=1$ and $p=\infty$ follow by approximation.)
Moreover, assume $u_t-v_t\in L^2(M)\cap L^p(M)$.
Then a slight modification of the proof of the previous proposition yields
\begin{align*}
&-\frac1p\partial_t \int |u_t-v_t |^p\,dm
 = -\int |u_t-v_t|^{p-1}\, \mbox{sign}(u_t-v_t)\, \partial_t (u_t - v_t) \,dm \\
&= \int D\left(|u_t-v_t|^{p-1}\, \mbox{sign}(u_t-v_t)\right)\, (\Nabla u_t - \Nabla v_t) \,dm \\
&= (p-1)\int |u_t-v_t|^{p-2}\, D(u_t-v_t)\, (\Nabla u_t - \Nabla v_t) \,dm \\
&\ge (p-1)\kappa_M\int |u_t-v_t|^{p-2}\, D(u_t-v_t)\cdot \Nabla (u_t -  v_t) \,dm \\
&= \frac4{p^2}(p-1)\kappa_M\int D\big(|u_t-v_t|^{p/2}\, \mbox{sign}(u_t-v_t)\big)\cdot
 \Nabla \big(|u_t -  v_t|^{p/2}\, \mbox{sign}(u_t-v_t)\big) \,dm \\
&\ge \frac4{p^2}(p-1)\kappa_M\chi_M\int \left|u_t-v_t\right|^{p} \,dm.
\end{align*}
To be rigorous, one should assume in the previous argumentation that $|u-v|$ is bounded
from above if $p>2$ or bounded away from $0$ if $p<2$, respectively.
To overcome this restriction, one can approximate $u$ and $v$ by bounded solutions
to the heat equation in the case $p>2$.
In the case $p<2$, one can approximate $|u-v|$ by $((u-v)^2+\varepsilon^2)^{1/2}$.

Obviously, the assumption $v\equiv0$ allows to replace $\kappa_M$ in the first inequality above by $1$.
The claim follows again by an application of Gronwall's lemma.
\end{proof}

Now let us switch from contraction estimates to integrated Gaussian estimates for the heat semigroup.
A preliminary step is the following:

\begin{lemma}\label{lm:psi}
Let $u$ be a global solution to the heat equation $\partial_t u 
=\bm{\Delta}u$ on $M$ and $\psi:M \rightarrow \R$ be a Lipschitz 
continuous function of bounded gradient $F(x,\Nabla \psi(x)) \le C$ 
for all $x \in M$. Then we have, for all $0\le s\le t$, 
\begin{equation}\label{davies}
\| e^{-\psi} u_t \|_{L^2} \le e^{C^2(t-s)} \| e^{-\psi} u_s \|_{L^2}.
\end{equation}
\end{lemma}

\begin{proof}
Straightforward calculations yield
\begin{align*}
&\frac{1}{2} \| e^{-\psi} u_t \|^2_{L^2} -\frac{1}{2} \| e^{-\psi} u_s\|^2_{L^2}
 = \int_s^t \int_M e^{-2\psi} u_r \partial_r u_r \,dm \,dr \\
&= -\int_s^t \int_M D(e^{-2\psi} u_r)(\Nabla u_r) \,dm \,dr \\
&= -\int_s^t \int_M \big\{ e^{-2\psi} Du_r(\Nabla u_r) -2u_r e^{-2\psi} D\psi(\Nabla u_r) \big\} \,dm \,dr \\
&\le \int_s^t \int_M \big\{ -F^2(\Nabla u_r) e^{-2\psi} +2F(\Nabla u_r) F^*(D\psi) |u_r| e^{-2\psi} \big\} \,dm \,dr \\
&\le \int_s^t \int_M F^2(\Nabla \psi) u_r^2 e^{-2\psi} \,dm \,dr
 \le C^2 \int_s^t \| e^{-\psi} u_r \|^2_{L^2} \,dr.
\end{align*}
Together with Gronwall's lemma, this implies the desired estimate.
\end{proof}

\begin{thm}[Integrated Gaussian Estimates \'a la Davies]\label{th:Dav}
For every $u,v \in L^2(M)$, we have
\begin{equation}
\int_M u P_tv \,dm \le \exp \bigg( -\frac{d^2(u,v)}{4t} \bigg) \|u\|_{L^2} \|v\|_{L^2},
\end{equation}
where $d(u,v)=\essi \{ d(x,y) \,:\, x\in\mathrm{supp}[u], \, y \in \mathrm{supp}[v] \}$.
\end{thm}

\begin{proof}
Given $u$ and $v$, apply Lemma \ref{lm:psi} to the function $\psi(x)=-Cd(x,v)$,
where $d(x,v) :=\essi \{ d(x,y) \,:\, y \in \mathrm{supp}[v] \}$ and $C>0$ is a constant to be fixed below.
Then
\begin{align*}
\int_M u P_t v \,dm &\le \| e^{-\psi} P_tv \|_{L^2} \| e^{\psi}u \|_{L^2}
 \le e^{C^2t}\| e^{-\psi} v \|_{L^2} \| e^{\psi} u \|_{L^2} \\
&\le e^{C^2t-Cd(u,v)} \|v\|_{L^2} \|u\|_{L^2}.
\end{align*}
Choosing $C=d(u,v)/2t$ now yields the claim.
\end{proof}

\section{The Heat Equation -- Local Solutions}\label{sc:local}

This chapter is devoted to studying the local regularity of solutions to the heat equation.
Formulation of results and proofs follow classical lines.
For the elliptic case, similar results have already been derived by Shen \cite{Shlap} and by
Belloni, Kawohl and Juutinen \cite{BKJ}. See also \cite{Di}, \cite{GS}.

Throughout the chapter, the assumptions $(\ref{eq:F-str2})$ and $(2.1)$ will be in force.

\begin{defn}
Given an open subset $\Omega\subset M$ and an open interval 
$I\subset\R$, we say that a real function $u$ on $I\times \Omega$ is 
a local solution to the heat equation $\partial_t u=\bm{\Delta}u$ on 
$I\times \Omega$ if $u\in L^2_{\mathrm{loc}}(I \times \Omega)$ with 
$F^*(Du) \in L^2_{\mathrm{loc}}(I\times\Omega)$ and for every 
smooth, compactly supported $v$ on $I\times\Omega$ $($or, 
equivalently, for  every $v \in H^1_c(I\times\Omega))$ 
\begin{equation}\label{local-heat}
\int_I\int_\Omega u_t \partial_t v_t \,dm \,dt =\int_I\int_\Omega Dv_t (\nabla u_t) \,dm\,dt.
\end{equation}
\end{defn}

\begin{remark}\rm
A function $u$ being a local solution to the heat equation implies that
$C_1u+C_2$ is a  local solution for every $C_1\in\R_+$ and every $C_2\in\R$.
In particular, constants are local solutions to the heat equation.
In general, it will not imply that $-u$ is
a local solution.
\end{remark}

\begin{example}\label{heat sol}\rm
Let $\| \cdot \|$ be any smooth, strictly convex Minkowski norm on $\R^n$,
put $F(x,\cdot)=\| \cdot \|$ for all $x$ and choose $m$ to be the Lebesgue measure.
Then for each fixed $y\in\R^n$ the function
\begin{equation}\label{gauss}
u(t,x)=t^{-n/2}\exp(-\|y-x\|^2/4t)\end{equation} is a local solution 
to the heat equation $\partial_t u=\bm{\Delta}u$ on 
$\R_+\times\R^n$. More generally, $u(t,x)=f(t,\|y-x\|)$ is a local 
solution to the heat equation for each smooth function 
$f:\R_+^2\to\R$ satisfying $\partial_r f(t,r) \le 0$ and 
\begin{equation}\label{rad-lapl}
\partial^2_rf(t,r)+\frac{n-1}{r}\partial_rf(t,r)=\partial_tf(t,r),\qquad \partial_rf(t,0)=0.
\end{equation}

If $f$ satisfies $\partial_rf(t,r)\ge0$ and $(\ref{rad-lapl})$, then the function $v(t,x)=f(t,\|x-y\|)$
is a local solution to the heat equation.
If $\| \cdot \|$ is even a norm (i.e., if in addition it is symmetric),
then the latter holds true without any restriction on the sign of $\partial_rf(t,r)$.
\end{example}

Note that the function $u$ in $(\ref{gauss})$ is $\mathcal{C}^2$ in
the space variable at $x=y$ if and only if $\| \cdot \|$ is a Hilbert norm.

\begin{prop}[Harnack Inequality]\label{prop:harnack}
Every local solution to the heat equation $\partial_t 
u=\bm{\Delta}u$ on $I\times\Omega$ is H\"older continuous $($more 
precisely, it is almost everywhere equal to a H\"older continuous 
function$)$. 

Continuous local solutions satisfy the parabolic Harnack inequality and the strong maximum principle.
\end{prop}

\begin{proof}
Since for given $u$ the Finsler Laplacian $\bm{\Delta}u$ coincides 
with the weighted Laplacian $\Delta^{(u)}$ in the Riemannian metric 
derived from $Z=\nabla u$ (Lemma \ref{lm:Lap}) and since for varying 
(and time-dependent) $u$ all these possible operators $\Delta^{(u)}$ 
are `locally uniformly elliptic', the claim is an immediate 
consequence of Saloff-Coste's result \cite{SC} for locally uniformly 
elliptic operators on weighted Riemannian manifolds. 
\end{proof}

\begin{prop}
The distributional time derivative $w=\partial_t u$ of any 
continuous local solution to the heat equation $\partial_t 
u=\bm{\Delta}u$ on $I\times\Omega$ lies in $H^1_{\mathrm{loc}}(M)$ 
and admits a H\"older continuous version $($which satisfies the 
parabolic Harnack inequality and the strong maximum principle$)$. It 
is a weak solution to the linear parabolic PDE 
\[ \partial_t w=\text{\emph{div}}(g^{*(u)} Dw) \]
with the locally uniformly elliptic, time dependent matrix $g^{*(u)}={g^{(u)}}^{-1}$ defined in $(\ref{g-u})$.
\end{prop}

\begin{proof}
We postpone the technical proof for the fact $\partial_t u\in H^1_{\mathrm{loc}}(M)$
to Appendix \ref{ap:2} and take this fact now for granted.
Let $\Phi$ be a smooth, compactly supported test function on
$I\times\Omega$. Applying (\ref{local-heat}) to $v=\partial_t\Phi$ yields
\begin{align*}
\int\int (\partial_t\Phi) w\,dm\,dt &= \int\int v \partial_t u\,dm\,dt \\
&=-\int\int Dv \cdot J^*(Du)\,dm\,dt= \int\int D\Phi\cdot \partial_t[J^*(Du)] \,dm\,dt \\
&=\int\int D\Phi \cdot g^*(Du) \cdot D(\partial_tu) \,dm\,dt
 =\int\int D\Phi \cdot g^{*(u)} \cdot Dw\,dm\,dt.
\end{align*}
Hence, $w$ is a weak solution to the linear PDE. Regularity theory
for solutions to linear second order PDEs now implies that $w$ has a
H\"older continuous version satisfying Harnack's inequality and
strong maximum principle.
\end{proof}

\medskip

In order to obtain higher order regularity results, we have to impose certain minimal smoothness assumptions
on the data $F$ and $m$.
We will assume that the maps $J^*(x,\alpha)$ and the logarithmic derivative
$-V(x)=\log[m(dx)/dx]$ of the measure $m$ are Lipschitz continuous in $x$.
More precisely,
\begin{description}
\item[we assume from now on]
that for each point $\overline{x}\in M$ there exists a local
coordinate system $(x^i)_{i=1}^n$ on a suitable neighborhood $U$ of
$\overline{x}$ and a number $\Lambda$ such that
\begin{equation}\label{smooth Finsler}
|\gamma_{ki}^*(x,\alpha)| \le \Lambda F^*(x,\alpha), \qquad |\eta_k(x)| \le \Lambda
\end{equation}
for almost all $x\in U$ and  all $\alpha \in T^*_x M$.
Here and henceforth
\[ \gamma_{ki}^*(x,\alpha) := \frac{\partial}{\partial x^k}\frac{\partial}{\partial\alpha^i}
\bigg( \frac{1}{2} F^{*2}(x,\alpha) \bigg)=\frac{\partial}{\partial x^k}J_i^*(x,\alpha) \]
and $\eta_k(x):=\partial V/\partial x^k(x)$ where $m(dx)=e^{-V(x)}dx_1 \cdots dx_n$.
\end{description}

The first important consequence of these assumptions is

\begin{thm}[$H^2$-Regularity]\label{thm:H2-reg}
Assume that the transfer maps $J^*$ as well as the logarithmic 
density of the measure $m$ are differentiable in $x$ as specified in 
$(\ref{smooth Finsler})$. Then every continuous local solution to 
the heat equation $\partial_t u=\bm{\Delta}u$ on $I\times\Omega$ is 
$H^2_{\mathrm{loc}}$ in $x$. 
\end{thm}

We postpone the technical proof to Appendix \ref{ap:3} and continue
with the proof of H\"older continuity of the derivatives of $u$.

\begin{lemma} For each local solution $u$ to the heat equation and each $k=1,\ldots,n$, the
partial derivative $w(t,x)=D_ku(t,x)= \partial u/\partial x^k(t,x)$ is a weak solution to the equation
\begin{equation}\label{space-der}
\partial_tw=\text{\emph{div}}(g^{*(u)} \cdot Dw) + \text{\emph{div}}\,H +h
\end{equation}
with a vector field $H \in L^2_{\mathrm{loc}}(I\times\Omega)$
and a function $h \in L^\infty_{\mathrm{loc}}(I\times\Omega)$ given by
\[ H_i(t,x)=\gamma_{ki}^* \big( x,Du(t,x) \big) -\eta_k(x) J_i^*\big( x,Du(t,x) \big) \]
and
\[ h(t,x)=\eta_k(x) \partial_t u(t,x). \]
\end{lemma}

\begin{proof}
Let a smooth, compactly supported test function $\Phi$ on $I\times\Omega$ be given.
Without restriction, we may assume that there exists a global coordinate system $(x^i)_{i=1}^n$
on $\Omega$ (or at least on the support of $\Phi$).
In these coordinates, let $m$ be given as $m(dx)=e^{-V(x)}dx^1 \cdots dx^n$.

Applying (\ref{local-heat}) to $v=D_k\Phi$ yields
\begin{align*}
&\int\int\partial_t(D_k\Phi) u \,dm\,dt
 = \int\int D(D_k\Phi) \cdot J^*(Du)\,dm\,dt \\
&= \int\int D\Phi \cdot \big[ -D_k\big( J^*(Du) \big) +(D_kV) J^*(Du) \big] \,dm\,dt \\
&= \int\int D\Phi \cdot[ -g^*(Du) \cdot D(D_ku) -(\gamma_{k\cdot}^*)(Du) +(D_kV) J^*(Du)] \,dm\,dt \\
&= -\int\int D\Phi \cdot[g^*(Du) \cdot Dw +H] \,dm\,dt.
\end{align*}
On the other hand,
\begin{align*}
\int\int \partial_t(D_k\Phi) u \,dm\,dt
&= \int\int [-(\partial_t \Phi) (D_k u) +(D_kV)(\partial_t\Phi) u] \,dm\,dt \\
&= -\int\int [(\partial_t \Phi)w+\Phi h] \,dm\,dt.
\end{align*}
That is,
\[ \int\int D\Phi \cdot[g^*(Du) \cdot Dw +H] \,dm\,dt
 =\int\int [(\partial_t\Phi)w +\Phi h] \,dm\,dt \]
for all smooth compactly supported $\Phi$ on $I\times \Omega$ and
thus
\[ \partial_tw=\text{\rm{div}}(g^{*(u)} \cdot Dw +H) +h \]
locally in distributional sense on $I\times\Omega$.
\end{proof}

\begin{lemma}\label{lm:wHh}
\begin{itemize}
\item[\rm{(i)}]
If $w\in L^p_{\mathrm{loc}}(I\times\Omega)$ is a weak solution to the equation $(\ref{space-der})$ with a vector field
$H\in L^p_{\mathrm{loc}}(I\times\Omega)$ for some $p\in [1,\infty]$ and a function
$h\in L^\infty_{\mathrm{loc}}(I\times\Omega)$, then $w \in L^q_{\mathrm{loc}}(I\times\Omega)$ for $q=pn/(n-2)$.

\item[\rm{(ii)}]
If $w$ is a weak solution to the equation $(\ref{space-der})$ with a vector field
$H\in L^p_{\mathrm{loc}}(I\times\Omega)$ for some $p>n+2$
and a function $h\in L^\infty_{\mathrm{loc}}(I\times\Omega)$, then $w$ is H\"older continuous.
\end{itemize}
\end{lemma}

\begin{proof}
(i) This result should be well known (perhaps even in a sharper version).
Since we could not find a reference, we include a sketch of the proof.
We do not discuss smoothing and cut-off arguments.
For simplicity, we assume that $I\times\Omega=(0,T)\times M$ and that $M$ is compact.

Let $w\in L^p_{\mathrm{loc}}(I\times\Omega)$ be a weak solution to the equation $(\ref{space-der})$
with a vector field $H\in L^p_{\mathrm{loc}}(I\times\Omega)$ and a function $h\in L^\infty_{\mathrm{loc}}(I\times\Omega)$.
Choose $w^{p-1}$ as a test function.
Then the weak formulation of $(\ref{space-der})$ implies
\begin{align*}
\frac1p\|w_T\|_{L^p}^p-\frac1p\|w_0\|_{L^p}^p
&=\int_0^T \frac1p\partial_t\|w_t\|_{L^p}^p\,dt=\int_0^T\int_M w^{p-1}\partial_tw\,dm\,dt \\
&=\int\int [-D(w^{p-1}) \cdot g^{*(u)} \cdot Dw - D(w^{p-1}) \cdot H +w^{p-1}h] \,dm\,dt \\
&\le -\frac{p-1}2 \int\int w^{p-2}Dw \cdot g^{*(u)} \cdot Dw \,dm\,dt \\
&\qquad +\frac{p-1}2 \int\int w^{p-2}H \cdot g^{(u)} \cdot H \,dm\,dt +\int\int w^{p-1}h \,dm\,dt \\
&\le -\frac{2(p-1)}{p^2} \|F^{*(u)}(D w^{p/2})\|_{L^2}^2 \\
&\qquad +\frac{p-1}2\|w\|_{L^p}^{p-2} \|F^{(u)}(H)\|_{L^p}^2 +\|w\|_{L^p}^{p-1} \|h\|_{L^\infty} \\
&\le C <\infty
\end{align*}
according to our assumptions on $u$, $H$ and $h$.
From this estimate, we first of all deduce that $\|w_t\|_{L^p}$ is bounded in $t$ on $I$.
Having this at hand, we secondly deduce that
\[ \|F^{*(u)}(D w^{p/2})\|_{L^2}<\infty. \]
Classical Sobolev inequality now implies $w^{p/2}\in L^{2^*}$ with $2^*=2n/(n-2)$.
That is, $w\in L^q$ with $q=pn/(n-2)$.

(ii) This is a standard estimate.
In the required version it can be found in \cite{SC}.
However, similar versions certainly had been known much earlier, e.g., in the works of Moser, Aronson and Serrin.
\end{proof}

\begin{thm}[$\mathcal{C}^{1,\alpha}$-Regularity] \label{thm:holder-reg}
Assume that the transfer maps $J^*$ as well as the logarithmic 
density of the measure $m$ are differentiable in $x$ as specified in 
$(\ref{smooth Finsler})$. Then every continuous local solution to 
the heat equation $\partial_t u=\bm{\Delta}u$ on $I\times\Omega$ is 
$\mathcal{C}^{1,\alpha}$ in $t$ and $x$. 
\end{thm}

\begin{proof}
To deduce the H\"older continuity, we apply the first assertion
of Lemma \ref{lm:wHh} to each of the partial derivatives $w=D_ku$ of
the given solution $u$. It implies that $w\in L^p_{\mathrm{loc}}$ for some
$p>2$ and thus in turn $H\in L^p_{\mathrm{loc}}$ (according to our
assumptions $(\ref{smooth Finsler})$ on the coefficients of the Finsler structure).
Finitely many iterations of this argument yield $w\in L^q_{\mathrm{loc}}$ for $q$ sufficiently large
in order to apply the second assertion of Lemma \ref{lm:wHh} which then implies H\"older continuity.
\end{proof}

\begin{remark} \rm
If $F$ is a smooth Finsler structure and if the logarithmic density 
of the measure $m$ is $\mathcal{C}^\infty$, then local solutions $u$ 
of the heat equation $\partial_t u=\bm{\Delta}u$ are 
$\mathcal{C}^\infty$ in $t$ and $x$ outside the set $\{(t,x) \,:\, 
Du(t,x)=0\}$. On this set, however, the solutions typically will not 
be $\mathcal{C}^2$. See Example \ref{heat sol}. 
\end{remark}

\section{Ricci Curvature and Heat Equation}\label{sc:Ricci}

From now on, we always assume that $M$ is  compact and that $F$ is smooth on $TM \setminus \{0\}$.
This in particular implies that the uniform ellipticity condition formulated in Chapter \ref{sc:pre} is equivalent
to the {\it strong convexity} of $F(x,\cdot)$ at every $x \in M$
(in the sense that the matrix $g_{ij}(x,\xi)$ in $(\ref{eq:F-str1})$
is positive-definite for all $\xi \in T_xM \setminus \{0\}$).

\medskip

We review some geometric concepts in a heuristic way, intended  for nonspecialists.
For further reading and more details, we refer to \cite{BCS} and \cite{Shlec}.

A $\mathcal{C}^1$-curve $\gamma:[0,l] \rightarrow M$ is called a {\it geodesic} if it has constant speed
(i.e., $F(\gamma,\dot{\gamma})$ is constant) and if it is locally minimizing,
i.e., given $t \in [0,l]$, there is $\varepsilon>0$ such that
$d(\gamma(s),\gamma(s')) =\int_s^{s'} F(\gamma,\dot{\gamma}) \,d\tau$ holds
for all $s,s' \in [0,l] \cap [t-\varepsilon,t+\varepsilon]$ with $s<s'$.
Such $\gamma$ is in fact $\mathcal{C}^{\infty}$ and, for any $x \in M$ and $y \in M$ sufficiently close to $x$,
there is a unique {\it minimal} geodesic $\gamma:[0,1] \rightarrow M$ with $\gamma(0)=x$ and $\gamma(1)=y$
(i.e., $d(x,y)=\int_0^1 F(\gamma,\dot{\gamma}) \,d\tau$).

Given $x \in M$ and $\xi \in T_xM$, we define the
{\it exponential map} by $\exp_x \xi:=\gamma(1)$ provided there exists a geodesic $\gamma:[0,1] \rightarrow M$
with $\gamma(0)=x$ and $\dot{\gamma}(0)=\xi$.
By the Hopf-Rinow theorem (cf.\ \cite[Theorem 6.6.1]{BCS}), $(M,F)$ is forward complete
if and only if $\exp_x$ is defined on all of $T_xM$ for each (or some) $x \in M$.
In this case, any two points $x,y \in M$ can be connected by a minimal geodesic from $x$ to $y$.

For a unit vector $v \in T_xM$, let $r(v) \in (0,\infty]$ be the supremum of $r>0$ such that
the geodesic $t \mapsto \exp_x tv$ is minimal on $[0,r]$.
If $r(v)<\infty$, then $\exp_x (r(v)v)$ is called a {\it cut point} of $x$,
and the {\it cut locus} $\mathrm{Cut}(x)$ of $x$ is defined as the set of all cut points of $x$.
The exponential map $\exp_x$ is a $\mathcal{C}^{\infty}$-diffeomorphism from $\{ tv \,:\, v \in T_xM,\, F(v)=1,\, t \in (0,r(v)) \}$
to $M \setminus (\mathrm{Cut}(x) \cup \{ x \})$.

\medskip

Fix a unit vector $\xi \in T_xM$ (i.e., $F(x,\xi)=1$) and let $Z$ be an arbitrary $\mathcal{C}^{\infty}$-vector field
on an open neighborhood $U$ of $x$ with $Z(x)=\xi$ and such that every integral curve of $Z$ is a geodesic.
A typical example is $Z=\Nabla [d(\gamma(-\varepsilon),\cdot)]$ for sufficiently small $\varepsilon>0$,
where $\gamma:[-\varepsilon,\varepsilon] \rightarrow M$ is a geodesic with $\dot{\gamma}(0)=\xi$.
Then $Z$ induces the Riemannian structure $g_Z(x):=g(x,Z(x))$ on $U$ through $(\ref{eq:F-str1})$ (see also $(\ref{eq:gxi})$),
and the {\it flag curvature} $\mathcal{K}(\xi,\eta)$ of $\xi$ and a linearly independent unit vector $\eta \in T_xM$
is defined as the sectional curvature of the plane spanned by $\xi$ and $\eta$ with respect to $g_Z$
(see \cite[Proposition 6.2.2]{Shlec}).
Similarly, the {\it Ricci curvature} $\mathrm{Ric}(\xi)$ is the Ricci curvature of $\xi$ with respect to $g_Z$.

Recall our arbitrarily fixed measure $m$ on $M$ and its representation $m(dx)=e^{-V(x)}dx^1 \cdots dx^n$ (see $(\ref{measure})$).
Similarly,  the Riemannian volume element $m_Z$ induced from $g_Z$ has a representation  as
\[ m_Z(dx) =e^{-W_Z(x)} dx^1 \cdots dx^n \]
for some function $W_Z$ on $U$.
Thus we can represent $m(dx)=e^{-V_Z(x)}m_Z(dx)$ with $m_Z$ as a reference measure
and $V_Z=V-W_Z$ as a weight function. We put
\begin{equation}\label{eq:dV}
\partial_\xi V_Z= \frac{d}{dt}\Big|_{t=0} V_Z\big( \gamma(t) \big), \qquad
\partial^2_{\xi} {V}_Z=\frac{d^2}{dt^2}\Big|_{t=0} V_Z\big( \gamma(t) \big),
\end{equation}
where $\gamma:[-\varepsilon,\varepsilon] \rightarrow M$ is the geodesic with $\dot{\gamma}(0)=\xi$.
The important observation now is that for given $\xi$ the quantities
$\mathrm{Ric}(\xi):=\mathrm{Ric}_{g_Z}(Z,Z)$ as well as $\partial_{\xi} {V}_\xi:=\partial_\xi V_Z$
and $\partial^2_{\xi} {V}_\xi:= \partial^2_{\xi} {V}_Z$ do not depend on the choice of the vector field $Z$
(provided it has geodesics as integral curves).

The following lower Ricci curvature bound was introduced in \cite{Oint} inspired by the theory of weighted Riemannian manifolds.

\begin{defn}\label{df:NRic} \rm
Let $(M,F,m)$ be a smooth, $n$-dimensional Finsler manifold endowed with a smooth measure $m$ and let $K \in \R$.
\begin{itemize}
\item[(i)]
We say that $(M,F,m)$ satisfies the bound $n$-$\mathrm{Ric} \ge K$
if $\mathrm{Ric}(\xi) \ge K$ and $\partial_{\xi} {V}_\xi=0$ for any unit vector $\xi \in T_xM$.

\item[(ii)]
We say that $(M,F,m)$ satisfies the bound $N$-$\mathrm{Ric} \ge K$ for some given number $N \in (n,\infty)$ if
\[ \mathrm{Ric}_N(\xi):= \mathrm{Ric}(\xi) +\partial^2_{\xi}{V}_\xi -\frac{(\partial_{\xi} {V}_\xi)^2}{N-n} \ge K \]
for any unit vector $\xi \in T_xM$.

\item[(iii)]
We say that $(M,F,m)$ satisfies the bound $\infty$-$\mathrm{Ric} \ge K$
if $\mathrm{Ric}_{\infty}(\xi):=\mathrm{Ric}(\xi) +\partial^2_{\xi} {V}_\xi \ge K$  for any unit vector $\xi \in T_xM$.
\end{itemize}
\end{defn}

The infinite dimensional case (iii) corresponds to the Bakry-\'Emery tensor (\cite{BE})
and the finite dimensional case (ii) is an analogue of Qian's generalized one (\cite{Qi}, see also \cite{Lo}).
The most restricted case (i) still admits a number of non-Riemannian spaces.
For instance, the Busemann-Hausdorff measure on a Finsler manifold of Berwald type satisfies
$\partial {V} \equiv 0$ (\cite[Propositions 2.6, 2.7]{Shvol}).
However, the existence of a measure satisfying $\partial {V} \equiv 0$ should be a strong constraint
among general Finsler manifolds, and then there is no advantage in dealing with concrete measures.
This is the reason why we consider an arbitrary measure $m$ on $M$.

\medskip

\begin{thm}\label{th:Lcom}
Assume that $N$-$\mathrm{Ric}\ge K$ for some pair  $K, N\in \R$ with $N \ge \dim M$.
Then the Laplacian of the distance function $u(x)=d(z,x)$ from any given point $z\in M$ can be estimated as follows
\begin{equation} \bm{\Delta} u(x)\le \sqrt{-(N-1)K}\cdot \coth\bigg( \sqrt{\frac{- K}{N-1}}d(z,x) \bigg)
\end{equation}
pointwise on $M_z:=M\setminus (\{z\}\cup \mathrm{Cut}(z))$ and in the sense of distributions on $M\setminus\{z\}$.
If $K=0$, then the RHS should be interpreted as $(N-1)/d(z,x);$
if $K>0$, then as $\sqrt{(N-1)K}\cdot \cot (\sqrt{K/(N-1)}d(z,x))$.

\end{thm}

\begin{proof}
Let us fix $z\in M$ and put $u(x)=d(z,x)$. Then outside of $M_z$ the 
vector field $Z(x):=\Nabla u(x)$ is well-defined, smooth and 
satisfies $F(x,Z(x))=1$. Let $d_Z$ and $\Delta_Z$ denote the 
Riemannian distance and the weighted Laplacian on $M_z$ with the 
Riemannian metric $g_Z(x):=g(x,Z(x))$. Then $u(x)=d_Z(z,x)$ and 
$\bm{\Delta}u(x)=\Delta_Zu(x)$ by Lemma \ref{lm:Lap}. Hence, 
estimating the Finsler Laplacian of the Finsler distance amounts to 
estimating the weighted Riemannian Laplacian of the Riemannian 
distance function. 

Due to our curvature assumption on the Finsler space $(M,F,m)$,
the weighted Riemannian space $(M_z,g_z,m)$ satisfies the curvature bound  $N$-$\mathrm{Ric}\ge K$ in the sense of Definition \ref{df:NRic}.
On weighted Riemannian spaces, the latter is known to be equivalent to a {\it generalized Bochner inequality}
or {\it $\Gamma_2$-inequality in the sense of Bakry-\'Emery}
\begin{equation}\label{Gamma2a}
\Gamma_2(v,v)\ge\frac1N(\Delta_Zv)^2+K\cdot\Gamma(v,v)
\end{equation}
for all smooth functions $v$ on $M_z$.
Here $\Gamma(v,w)=Dv(\nabla_Z w)$ and
\[ \Gamma_2(v,v)=\frac12 \Delta_Z\Gamma(v,v)-\Gamma(\Delta_Zv,v), \]
see \cite{BE}, \cite{Qi}, \cite{Lo}.
The remarkable observation of Bakry and Qian \cite{BQ} is the `self-improving property'
of $(\ref{Gamma2a})$ saying that the validity of the previous estimate (for all smooth $v$) entails
the stronger estimate
\[ \Gamma_2(v,v) \ge \frac1N(\Delta_Zv)^2+K\cdot\Gamma(v,v)
 +\frac{N}{N-1}\left[\frac{\Delta_Zv}N-\frac{\Gamma(v,\Gamma(v,v))}{2\Gamma(v,v)}\right]^2 \]
valid for all smooth functions $v$ with nonvanishing gradient.
Applying the latter to $u(x)=d(z,x)$ and using the fact that $\Gamma(u,u)=1$ yields
\begin{equation}\label{Gamma2c}
\Gamma_2(u,u)\ge\frac1{N-1}(\Delta_Zu)^2+K
\end{equation}
on $M_z$, where $\Gamma_2(u,u)=-D(\Delta_Zu)(\nabla_Zu)=-D(\Delta_Zu)(Z)$.

Now let $\gamma: [0,l)\to M$ be any minimizing, unit speed geodesic 
in $(M,F)$ emanating from $z$. Then $d(z,\gamma_t)=t$ and 
$\dot\gamma_t=Z(\gamma_t)$. Put $\phi_t=\bm{\Delta} u(\gamma_t)$ for 
$t \in (0,l)$. Then $(\ref{Gamma2c})$ together with Lemma 
\ref{lm:Lap} states 
\[ -\dot\phi_t\ge \frac1{N-1}(\phi_t)^2+K \]
on $(0,l)$. Comparison results for ODEs then imply
\begin{equation*}\label{phi-b}
\phi_t \le \sqrt{(N-1)K}\cdot \cot\left(\sqrt{\frac K{N-1}}(t+t_0)\right)
\end{equation*}
for some $t_0\le0$ (and the usual interpretation of the RHS if $K\le0$). Local asymptotic for small $t$ implies $t_0=0$.
This proves the claim on the pointwise estimate of the Laplacian on $M_z$.

\medskip

The extension to a distributional inequality, valid also on the cut locus, follows by the well-known Calabi argument.
\end{proof}

\begin{cor}\label{cr:Lcom}
Assume that $N$-$\mathrm{Ric}\ge K$ and let $u(x)=f(d(z,x))$ for some nondecreasing smooth function
$f: (0,\infty)\to \R$.
Then on $M_z$,
\begin{equation}\label{f-dist}
\bm{\Delta} u(x)\le f''\big( d(z,x) \big) +f'\big( d(z,x) \big) \sqrt{(N-1)K}
 \cdot \cot\bigg(\sqrt{\frac K{N-1}}d(z,x)\bigg)
\end{equation}
$($if $K>0$, with the appropriate modification on the right-hand side for $K \le 0)$.
Similarly, if $v(x)=h(d(x,z))$ for some nonincreasing smooth function $h:(0,\infty)\to\R$.
Then on $M_z$,
\begin{equation}\label{h-dist}
\bm{\Delta} v(x)\ge h''\big( d(x,z) \big) +h'\big( d(x,z) \big) \sqrt{(N-1)K}
 \cdot \cot\bigg(\sqrt{\frac K{N-1}}d(x,z)\bigg).
\end{equation}
In both cases, the estimates extend to hold in the sense of distributions on all of $M\setminus\{z\}$.

If the function $f$  has a smooth extension to $[0,\infty)$ with $f'(0)=0$,
then the inequality $(\ref{f-dist})$ holds on all of $M$ in the sense of distributions.
Analogously for $(\ref{h-dist})$ provided $h'(0)=0$.
\end{cor}

\begin{proof}
The first claim follows from Theorem \ref{th:Lcom} by simple application of the chain rule:
\[ \bm{\Delta} f(u) =\mathrm{div}\big( \Nabla f(u) \big) =\mathrm{div}(f'(u)\Nabla u)
 =f'(u) \bm{\Delta} u + f''(u) Du(\Nabla u) \]
and the fact that $Du(\Nabla u)=1$.

For the second claim, a similar argumentation with $v(x)=d(x,z)$ yields
\begin{align*}
\bm{\Delta} h(v) &=\mathrm{div}\big( \Nabla h(v) \big) =\mathrm{div}\big( -h'(v)\Nabla(-v) \big) \\
&=h'(v)\big(-\bm{\Delta}(-v) \big) +h''(v)\cdot D(-v) \big(\Nabla(-v) \big).
\end{align*}
Observing that $v(x)=\overleftarrow{d}(z,x)$, 
$D(-v)(\Nabla(-v))=Dv(\overleftarrow{\Nabla}v)=1$ and 
$(-\bm{\Delta}(-v))=\overleftarrow{\bm{\Delta}}v$, the claim follows 
as before since the bound $N$-$\mathrm{Ric}\ge K$ for $(M,F,m)$ 
implies the same bound for the Finsler space with reverse structure 
$(M,\overleftarrow{F},m)$. 

\medskip

It remains to prove that (\ref{f-dist}) holds at the origin in the 
sense of distributions provided $f'(0)=0$. Without restriction, we 
may assume $f(r)=r^2$. (Otherwise, choose smooth $g$ with 
$f(r)=g(r^2)$ and use chain rule.)
Obviously, for $u(x)=d^2(z,x)$, the distribution $\bm{\Delta}u$ assigns no mass to the origin.
(Choose $\psi_{\varepsilon}(x)=(1-\varepsilon^{-2}d^2(z,x))_+$ as test function.)
\end{proof}

\begin{cor}\label{cr:sub}
Assume that $N$-$\mathrm{Ric}\ge K$ and let $h=h(t,r)$ be a smooth solution to the PDE
\begin{equation}\label{pde-rad}
\partial_th=\partial_r^2h+\partial_r h \sqrt{(N-1)K}\cdot \cot\bigg(\sqrt{\frac K{N-1}}r\bigg)
\end{equation}
on $(0,\infty)\times(0,L)$ $($if $K>0$, with the appropriate modification on the right-hand side for $K \le 0)$,
where $L=\pi\sqrt{(N-1)/K}$ if $K>0$ and 
$L=\infty$ else. Assume in addition $\partial_rh\le0$ on 
$(0,\infty)\times (0,L)$ and $\partial_rh=0$ on $(0,\infty)\times 
\{0\}$. Then for any $z\in M$ the function $u(t,x)=h(t,d(x,z))$ is a 
subsolution to the heat equation on $M$. That is, $\partial_t 
u\le\bm{\Delta}u$ in the sense of distributions on $(0,\infty)\times 
M \setminus \{z\}$. 
\end{cor}

\begin{example} \rm
(i) Assume that $N$-$\mathrm{Ric}\ge 0$.
Then for any $z\in M$ the function
\[ u(t,x)=t^{-N/2} \exp\bigg(-\frac{d^2(x,z)}{4t}\bigg) \]
is a subsolution to the heat equation on $M$.

(ii) Assume that $3$-$\mathrm{Ric}\ge -2$.
Then for any $z\in M$ the function
\[ u(t,x)=t^{-3/2} \frac{d(x,z)}{\sinh(d(x,z))} \exp\bigg(-t-\frac{d^2(x,z)}{4t}\bigg) \]
is a subsolution to the heat equation on $M$.
\end{example}

\begin{thm}[Cheeger-Yau Estimate]\label{th:CY}
Assume $N$-$\mathrm{Ric}\ge K$ for some pair $K,N \in \R$ with $N \ge \dim M$ and let $u$ be a solution to the
heat equation on $[0,\infty)\times M$ with $u(0,\cdot)\ge h_0(d(\cdot,z))$ for some $z\in M$
and some smooth decreasing function $h_0$ on $[0,L)$.
Then
\begin{equation}
u(t,x)\ge h^{K,N}\big( t,d(x,z) \big)
\end{equation}
for all $t>0$ and $x\in M$ where $h^{K,N}$ denotes the solution to the PDE $(\ref{pde-rad})$
with initial condition $h^{K,N}(0,\cdot)=h_0$ and Neumann boundary condition $\partial_r h^{K,N}(\cdot,0)=0$.
\end{thm}

\begin{proof}
We first observe that $\partial_r h_0\le0$ implies $\partial_r h^{K,N}(t,\cdot)\le0$ for all $t>0$.
Then the claim follows from the parabolic maximum principle along with Corollary \ref{cr:sub}.
\end{proof}

Next, we are going to apply the above estimate to the `fundamental solution' for the heat equation on $M$.
What we have in mind is to study $p_t(x,z)=P_t \delta_z(x)$,
the solution to the heat equation with initial data $\delta_z$.
Unfortunately, $P_t\delta_z$ is not defined since our heat semigroup only acts on $L^2(M)$
(or on $\bigcup_{1 \le p \le \infty}L^p(M)$, see Theorem \ref{thm:cont}),
but -- until now -- not on measures.
We thus will define $p_t(x,z)$ via approximation of the initial data $\delta_z$.

For this purpose, let
\[ \rho(z)=\lim_{r\to 0}\frac{m(B^-(z,r))}{c_n\cdot r^n}\]
with $c_n:=\pi^{n/2}/\Gamma(n/2+1)$ being the volume of the $n$-dimensional Euclidean unit sphere.
Recall that $B^-(z,r)=\{x\in M: d(x,z)<r\}$ denotes the backward open ball in $M$.
Given $K\in \R$ and $n\in \N$ let $p^{K,n}_t(r)$ denote the unique solution of the above PDE (\ref{pde-rad})
with $p_t^{K, n}(r)dr\to \delta_0(dr)$ weakly as $t\to0$.
Recall that for each fixed $\zeta$ in the model space $\mathbb{M}^{K,n}$ of dimension $n$ and constant sectional curvature $K/(n-1)$
the function $(t,\xi)\mapsto p_t^{K,n}(d(\xi,\zeta))$ is a solution of the heat equation on $\mathbb{M}^{K,n}$.

\begin{thm}
Assume that the Finsler space $(M,F,m)$ is compact and satisfies $n$-$\mathrm{Ric}\ge K$ for some $K \in \R$
$($with $n$ being the dimension of $M)$.

\begin{itemize}
\item[{\rm (i)}]
For all $t>0$ and all $x,z\in M$
\[ p_t(x,z):=\frac1{\rho(z)}\, \lim_{\varepsilon\to0}P_{t-\varepsilon} u_{\varepsilon}(x) \]
exists as a monotone limit with $u_\varepsilon(x):=p^{K,n}_{\varepsilon}(d(x,z))$.

\item[{\rm (ii)}]
For each $z\in M$ the function $(t,x)\mapsto p_t(x,z)$ is a solution to the heat equation on $(0,\infty) \times M$
with $p_t(x,z)m(dx)\to\delta_z(dx)$ weakly in the sense of measures as $t\to0$.

\item[{\rm (iii)}]
For all $t>0$ and all $x,z\in M$
\[ p_t(x,z)\ge\frac1{\rho(z)} \, p^{K,n}_t\big( d(x,z) \big).\]
\end{itemize}
\end{thm}

\begin{proof}
Throughout the proof we fix $K$ and $z\in M$.
(i) According to the previous theorem
\begin{equation}\label{mon-fund}
P_{t-s}u_s(x)\ge p^{K,n}_t\big( d(x,z) \big)
\end{equation}
for all $0<s<t$ and all $x\in M$.
Hence, for all $0<r<s<t$
\[ P_{t-r}u_r(x)=P_{t-s}\left(P_{s-r}u_r\right)(x)\ge P_{t-s}\Big( p^{K,n}_s\big( d(\cdot,z)\big) \Big)(x)=P_{t-s}u_s(x).\]
This proves the monotonicity and thus the existence of the limit.

(iii) follows immediately from ($\ref{mon-fund}$) as $s\to0$.

(ii) Given $s>0$, for each $\varepsilon \in (0,s)$ the function $v_\varepsilon(t,x):=P_{t-\varepsilon}u_\varepsilon(x)$
is a nonnegative solution to the heat equation on $[s,\infty) \times M$.
Hence, in particular it satisfies the parabolic Harnack inequality and,
with $|\partial B^-(z,r)| :=(\partial/\partial r)m(B^-(z,r))$,
\[ \int_M v_\varepsilon (t,x)\, m(dx)=\int_M u_\varepsilon(x)\,m(dx)
 =\int_0^\infty p_\varepsilon^{K,n}(r)\cdot |\partial B^-(z,r)|\, dr \quad\to\quad\rho(z) \]
uniformly in $\varepsilon\in (0,s)$ as $s\to0$.
Thus the monotone convergence of $v_\varepsilon(t,x)$ together with the compactness of $M$ imply uniform convergence in $x$
as well as $L^2$-convergence (for each fixed $t> 2s$) as $\varepsilon\to0$.
Together with the $L^2$-contraction property of the heat semigroup this then yields that the limit is again a solution
to the heat equation on $(2s,\infty) \times M$.

The proof of the weak convergence follows easily from property (iii).
Indeed, for each continuous function $f$ on $M$, bounded in modulus by $C$, we obtain
\begin{align*}
\int f(x)p_t(x,z)\,m(dx) &= -C +\int (f(x)+C)p_t(x,z)\,m(dx)\\
&\ge -C+\int (f(x)+C)\frac1{\rho(z)}p_t^{K,n}\big( d(x,z) \big) \,m(dx)\ \to \ f(z)
\end{align*}
as $t\to0$.
Similarly, we deduce $\limsup_{t\to0}\int f(x)p_t(x,z)\,m(dx)\le f(z)$
which then proves the claim.
\end{proof}

\section{The Finsler Structure of the Wasserstein Space}\label{sc:Wass}

In this chapter, we introduce the Finsler structure of the Wasserstein space over a smooth, compact Finsler manifold.
This concept goes back to Otto's pioneering work for Euclidean spaces (\cite{Ot}).
Our discussion follows (\cite{Vi1} and) \cite[\S 8]{AGS} for Hilbert spaces and \cite{Vi2} for Riemannian manifolds as well.

We denote by $\mathcal{P}(M)$ the set of all Borel probability measures on $M$,
and $\mathcal{P}_{\mathrm{ac}}(M) \subset \mathcal{P}(M)$ stands for the subset consisting
of absolutely continuous measures with respect to $m$.
Given $\mu,\nu \in \mathcal{P}(M)$, we say that $\pi \in \mathcal{P}(M \times M)$ is a {\it coupling}
of $(\mu,\nu)$ if its marginals are $\mu$ and $\nu$.

\begin{defn} \rm
For $\mu,\nu \in \mathcal{P}(M)$, we define the {\it $L^2$-Wasserstein distance} $d_W(\mu,\nu)$ by
\[ d_W(\mu,\nu) :=\inf_{\pi} \bigg( \int_{M \times M} d^2(x,y) \,d\pi(x,y) \bigg)^{1/2}, \]
where the infimum is taken over all couplings $\pi \in \mathcal{P}(M \times M)$ of $(\mu,\nu)$.
A coupling $\pi$ of $(\mu,\nu)$ is said to be {\it optimal} if it attains the infimum above.
\end{defn}

Given nonnegative functions $\rho, \sigma \in L^2(M)$ with $\mu:=\rho m, \nu:=\sigma m \in \mathcal{P}_{\mathrm{ac}}(M)$,
consider the coupling $\pi$ of $(\mu,\nu)$ given by
$\pi=\mathrm{diag}_{\sharp}(\min\{ \mu,\nu \}) +(\mu-\nu)_+ \times (\nu-\mu)_+$,
where $\mathrm{diag}(x):=(x,x)$ for $x \in M$ and $(\mu-\nu)_+(A)=\max\{ \mu(A)-\nu(A),0 \}$
for each Borel set $A \subset M$.
Then we have
\begin{align*}
d_W(\mu,\nu) &\le \mathrm{diam}(M) \Big\{ \big[ (\mu-\nu)_+ \times (\nu-\mu)_+ \big](M \times M) \Big\}^{1/2} \\
&=\frac{\mathrm{diam}(M)}{2} \| \rho-\sigma \|_{L^1}
 \le \frac{\mathrm{diam}(M) m(M)^{1/2}}{2} \| \rho-\sigma \|_{L^2}.
\end{align*}
Hence, if a curve in $L^2(M) \cap \mathcal{P}_{\mathrm{ac}}(M)$ is (locally) Lipschitz continuous as a curve in $L^2(M)$,
then it is (locally) Lipschitz continuous also as a curve in $\mathcal{P}(M)$.
In particular, the heat flow constructed in Theorem \ref{th:glob} starting from $u_0 \in H^1(M)$ with $u_0m \in \mathcal{P}(M)$
is locally Lipschitz continuous on $(0,\infty)$ as a curve in $\mathcal{P}(M)$.

A function $\varphi:M \rightarrow \R$ is said to be {\it $d^2/2$-concave} if there is a function
$\psi:M \rightarrow \R$ such that
\[ \varphi(x) =\psi^{\bar{c}}(x) :=\inf_{y \in M} \{ d^2(x,y)/2 -\psi(y) \} \]
holds for all $x \in M$.
Here $\psi^{\bar{c}}$ is called the {\it $\bar{c}$-transform} of $\psi$.
We similarly define the {\it $c$-transform} $\varphi^c$ of $\varphi$ by
$\varphi^c(y):=\inf_{x \in M}\{ d^2(x,y)/2 -\varphi(x) \}$.
Then $\varphi \le (\varphi^c)^{\bar{c}}$ is always true and $\varphi$ is $d^2/2$-concave
if and only if $\varphi=(\varphi^c)^{\bar{c}}$.
Moreover, any $d^2/2$-concave function is Lipschitz continuous and twice differentiable a.e.\ (see \cite{Ouni}).

We say that $\varphi$ is {\it $d^2/2$-convex} if $-\varphi$ is $d^2/2$-concave.
Then the Brenier-McCann characterization of optimal transport states the following (see \cite{Oint}):

\begin{thm}
For any $\mu \in \mathcal{P}_{\mathrm{ac}}(M)$ and any $\nu \in \mathcal{P}(M)$,
there exists a unique $d^2/2$-convex function $\varphi:M \rightarrow \R$
$($up to an additive constant$)$ such that the map $T(x) :=\exp_x(\Nabla \varphi(x))$
is a unique optimal transport from $\mu$ to $\nu$ in the sense that $\pi:=(\mathrm{Id}_M \times T)_{\sharp} \mu$
is a unique optimal coupling of $(\mu,\nu)$.
Furthermore, the curve $(\mu_t)_{t \in [0,1]}$ given by $\mu_t=(T_t)_{\sharp}\mu$ with $T_t(x)=\exp_x(t\Nabla \varphi(x))$
is a unique minimal geodesic from $\mu$ to $\nu$.
\end{thm}

The next lemma is an analogue of the Riemannian one in \cite{Vi2}.

\begin{lemma}\label{lm:ccon}
There exists a positive constant $\varepsilon>0$ depending on $M$ such that,
if a $\mathcal{C}^2$-function $\varphi$ on $M$ satisfies
\begin{equation}\label{eq:ccon}
\sup_M |\varphi| <\varepsilon, \qquad \sup_M F\big( \Nabla (-\varphi) \big) <\varepsilon, \qquad
\frac{d^2}{dt^2} \Big|_{t=0} \big[ \varphi \circ \gamma(t) \big] \le \varepsilon
\end{equation}
along every unit speed geodesic $\gamma$, then $\varphi$ is $d^2/2$-concave.
\end{lemma}

\begin{proof}
Thanks to the compactness of $M$, there are constants $c,\delta>0$ such that
\[ \frac{d^2}{dt^2} \Big|_{t=0} \bigg[ \frac{1}{2}d^2\big( \gamma(t),y \big) \bigg] \ge c \]
holds for any $y \in M$ and unit speed geodesic $\gamma$ with $d(\gamma(0),y) \le \delta$
(see \cite[Remark 15.1.4]{Shlec} or \cite{Ouni}).
(To be precise, the above inequality holds in the weak sense if $\gamma(0)=y$.)
In particular, the backward open ball $B^-(y,\delta)$ is convex for any $y \in M$.
It costs no generality to assume $\delta \le 4$.
We put $\varepsilon:=\min\{ \delta^2/4,c/2 \}$ and suppose that a $\mathcal{C}^2$-function $\varphi$
satisfies the condition $(\ref{eq:ccon})$ for this $\varepsilon$.

For each $y \in M$, consider the function $f_y(x):=d^2(x,y)/2 -\varphi(x)$.
By construction, $f_y$ is strictly convex ($(d^2/dt^2)|_{t=0}(f_y \circ \gamma) \ge c/2$
along any unit speed geodesic $\gamma$ with $d(\gamma(0),y) \le \delta$).
Given $x \not\in B^-(y,\delta)$, we observe $f_y(x) \ge \delta^2/4 >\varphi(y) =f_y(y)$.
Hence $f_y$ attains its minimum at a unique point in $B^-(y,\delta)$.

Fix arbitrary $x \in M$ and put $y=\exp_x(\Nabla(-\varphi)(x))$.
Note that $d(x,y)<\varepsilon \le \delta$ by assumption.
Then we have $D(-d^2(\cdot,y)/2)(x)=D(-\varphi)(x)$ and hence $Df_y(x)=0$.
This implies that $x$ is the unique minimizing point of $f_y$, so that $\varphi^c(y)=f_y(x)=d^2(x,y)/2-\varphi(x)$.
Therefore we find $(\varphi^c)^{\bar{c}}(x) \le d^2(x,y)/2 -\varphi^c(y) =\varphi(x)$.
As the reverse inequality is always true, we obtain $(\varphi^c)^{\bar{c}}(x) =\varphi(x)$ for all $x \in M$,
which shows that $\varphi$ is $d^2/2$-concave.
\end{proof}

In particular, for fixed $\mu \in \mathcal{P}_{\mathrm{ac}}(M)$ and any $\mathcal{C}^2$-function $\varphi$,
the map $T(x):=\exp_x(\Nabla (-c\varphi)(x))$ is the unique optimal transport from $\mu$ to $T_{\sharp}\mu$
provided $c>0$ is sufficiently small.
Thus we arrive at the following notion of tangent and cotangent spaces.

\begin{defn}\label{df:Wtan} \rm
For each $\mu \in \mathcal{P}(M)$, we define
\begin{align*}
T_{\mu} \mathcal{P} &:= \overline{\{ \Phi=\Nabla \varphi \,:\, \varphi \in \mathcal{C}^{\infty}(M) \}}^{F_W(\mu,\cdot)}, \\
 T^*_{\mu} \mathcal{P} &:= \overline{\{ \alpha=D\varphi \,:\, \varphi \in \mathcal{C}^{\infty}(M) \}}^{F_W^*(\mu,\cdot)},
\end{align*}
where the closures are taken with respect to the Finsler structures (Minkowski norms) depending on $\mu$:
\begin{align*}
F_W(\mu,\Phi) &:=\left(\int_M F^2\big( x,\Phi(x) \big) \,\mu(dx) \right)^{1/2}, \\
F^*_W(\mu,\alpha) &:=\left(\int_M {F^*}^2\big( x,\alpha(x) \big) \,\mu(dx) \right)^{1/2}.
\end{align*}
\end{defn}
Note that here the completion may be equally understood as forward completion or backward completion.
Indeed, since by assumption (1.2) (or (1.3)) the norms $F(x,\cdot)$ and $\overleftarrow{F}(x,\cdot)$
are locally equivalent and we are now in a compact setting,
convergence of $\Phi_n$ to $\Phi$ in the sense of
$F_W(\mu,\Phi_n-\Phi)\to0$ is equivalent to convergence in the sense of $F_W(\mu,\Phi-\Phi_n)\to0$.
Similarly, elements of $T_{\mu}\mathcal{P}$ consist of equivalence classes of vector fields
$\Phi_1, \Phi_2$ with $F_W(\mu,\Phi_1-\Phi_2)=0$ or equivalently with $F_W(\mu,\Phi_2-\Phi_1)=0$.

Let us remark that $F_W$ and $F_W^*$ are dual to each other if we define a pairing between
$T^*_{\mu} \mathcal{P}$ and $T_{\mu} \mathcal{P}$ by
\[ \langle \alpha,\Phi \rangle_{\mu} :=\int_M \langle \alpha(x),\Phi(x) \rangle_x \,\mu(dx), \]
where $\langle \cdot,\cdot \rangle_x$ denotes the natural pairing between $T^*_xM$ and $T_xM$.
The Legendre transform $J^*_W(\mu,\cdot): T^*_{\mu} \mathcal{P} \rightarrow T_{\mu} \mathcal{P}$ is defined by
\[ \alpha=\big( x \mapsto \alpha(x) \big) \  \mapsto \ J^*_W(\mu,\alpha)=\big( x \mapsto J^*(x,\alpha(x)) \big). \]
Similarly to $J^*$, $J^*_W(\mu,\alpha)$ is the maximizer of the function
\[ \Phi \ \mapsto \ \langle \alpha,\Phi \rangle_{\mu} -\frac{1}{2} F_W^2(\mu,\Phi) -\frac{1}{2}F_W^{*2}(\mu,\alpha) \]
and $F_W(\mu,J_W^*(\mu,\alpha)) =F_W^*(\mu,\alpha)$.

\bigskip

Recall that the {\it relative entropy} $\mathrm{Ent}(\mu)$ of $\mu \in \mathcal{P}(M)$
is defined by
\[ \mathrm{Ent}(\mu):=\int_M \rho \log \rho \,dm \ \in \ (-\infty,\infty] \]
if $\mu=\rho m \in \mathcal{P}_{\mathrm{ac}}(M)$, and by $\mathrm{Ent}(\mu):=\infty$ otherwise.
According to \cite{StI} and \cite{LV1}, we say that $(M,F,m)$ satisfies the {\it curvature-dimension condition} $\mathsf{CD}(K,\infty)$
for some $K \in \R$ if the relative entropy is $K$-convex in the sense that
any $\mu,\nu \in \mathcal{P}(M)$ admit a minimal geodesic $(\mu_t)_{t \in [0,1]}$ from $\mu$ to $\nu$
such that
\[ \mathrm{Ent}(\mu_t) \le (1-t)\mathrm{Ent}(\mu) +t\mathrm{Ent}(\nu) -\frac{K}{2}(1-t)td_W(\mu,\nu) \]
holds for all $t \in [0,1]$.
A similar, but more involved convexity property is used to define the {\it curvature-dimension condition}
$\mathsf{CD}(K,N)$ for arbitrary real numbers $N\ge1$.

\begin{thm}[$N$-$\mathrm{Ric} \ge K$ equals $\mathsf{CD}(K,N)$, \cite{Oint}]\label{th:CD}
For a compact, smooth Finsler space $(M,F,m)$, the bound $\infty$-$\mathrm{Ric} \ge K$ in the sense of Definition $\ref{df:NRic}$
is equivalent to $\mathsf{CD}(K,\infty)$.
More generally, $N$-$\mathrm{Ric} \ge K$ in the sense of Definition $\ref{df:NRic}$
is equivalent to $\mathsf{CD}(K,N)$ in the sense of Lott and Villani {\rm \cite{LV2}} and Sturm {\rm \cite{StII}}.
\end{thm}

\medskip

We recall one striking application.

\begin{thm}[Lichnerowicz Inequality, \cite{Oint}]\label{th:Lich}
Let $(M,F,m)$ be a compact smooth  Finsler space satisfying the bound $N$-$\mathrm{Ric} \ge K$
for some $K>0$ and $N \in [n,\infty]$.
Then for any Lipschitz continuous function $u:M \rightarrow \R$ with $\int_M u \,dm=0$, we have
\[ \int_M u^2 \,dm \le \frac{N-1}{KN} \int_M F(\Nabla u)^2 \,dm. \]
In other words, with notations from $(\ref{eq:nuM})$,
\[ \overline\chi_M\ge K\frac N{N-1}. \]
In the case $N=\infty$, the constant on the RHS should be understood as $K$.
\end{thm}

\section{Heat Flow as Gradient Flow in the Wasserstein Space}\label{sc:Wgrad}

We continue our analysis of the Wasserstein space over a smooth, compact Finsler manifold.
Using the continuity equation $(\ref{eq:coeq})$ below, we will see that the heat flow with respect to the reverse Finsler
structure is regarded as the gradient flow of the relative entropy.
See \cite{JKO} for original work on Euclidean spaces and \cite{Ogra}, \cite{Sa} and \cite{Vi2}
for related work on various Riemannian spaces.

We first observe that the Wasserstein distance is actually interpreted as the distance associated with
the Finsler structure introduced in Definition \ref{df:Wtan}.
The next lemma is an analogue of \cite[Theorem 1.1.2]{AGS} with a slight modification
caused by the nonsymmetric distance.

\begin{lemma}\label{lm:metd}
For any locally Lipschitz continuous curve $(\mu_t)_{t \in I} \subset \mathcal{P}(M)$ on an open interval
$I \subset \R$, the $($forward$)$ metric derivative
\[ |\dot{\mu}_t| := \lim_{s \to t} \frac{d_W(\mu_{\min\{ s,t \}},\mu_{\max\{ s,t \}})}{|t-s|} \]
exists at a.e.\ $t \in I$.
Moreover, $|\dot{\mu}| \in L^{\infty}_{\mathrm{loc}}(I)$ and
$d_W(\mu_s,\mu_t) \le \int_s^t |\dot{\mu}_{\tau}| \,d\tau$ holds for all $s,t \in I$ with $s<t$.
\end{lemma}

\begin{proof}
Take a countable dense set $\{ \nu_n \} \subset \{ \mu_t \,:\, t \in I \} \subset \mathcal{P}(M)$
and define the function $d_n(t):=d_W(\nu_n,\mu_t)$.
Note that $d_n$ is locally Lipschitz continuous uniformly in $n$, so that the function
$D(t):=\sup_n d'_n(t)$ is well-defined a.e.\ on $I$ and $D \in L^{\infty}_{\mathrm{loc}}(I)$.
It follows from the triangle inequality that
\[ \liminf_{s \uparrow t} \frac{d_W(\mu_s,\mu_t)}{t-s}
 \ge \sup_n \liminf_{s \uparrow t} \frac{d_n(t)-d_n(s)}{t-s} =D(t) \]
for a.e.\ $t \in I$.
Moreover, we deduce from the density of $\{ \nu_n \}$ that
\[ d_W(\mu_s,\mu_t) =\sup_n \{ d_n(t)-d_n(s) \} =\sup_n \int_s^t d'_n \,d\tau \le \int_s^t D \,d\tau. \]
Therefore we have
\[ \lim_{s \uparrow t} \frac{d_W(\mu_s,\mu_t)}{t-s} =D(t) \]
for a.e.\ $t \in I$.
We similarly obtain $\lim_{s \downarrow t} d_W(\mu_t,\mu_s)/(s-t) =D(t)$ for a.e.\ $t \in I$
and this completes the proof.
\end{proof}

\begin{lemma}\label{lm:tanv}
Let $I \subset \R$ be an open interval and $(\mu_t)_{t \in I} \subset \mathcal{P}(M)$ be
a locally Lipschitz continuous curve.
Suppose that a Borel vector field $\Phi(t,x) \in T_xM$ on $I \times M$
with $F(\Phi) \in L^2_{\mathrm{loc}}(I \times M, d\mu_tdt)$ satisfies the continuity equation
\[ \partial_t \mu_t +\mathrm{div}(\Phi_t \mu_t) =0 \]
in the weak sense that
\begin{equation}\label{eq:coeq}
\int_I \int_M \{ \partial_t \psi_t +D\psi_t(\Phi_t) \} \,d\mu_t \,dt =0
\end{equation}
for all $\psi \in \mathcal{C}^{\infty}_c(I \times M)$,
where $\Phi_t:=\Phi(t,\cdot)$ and $\psi_t:=\psi(t,\cdot)$.
Then we have $F_W(\mu_t,\Phi_t) \ge |\dot{\mu}_t|$ for a.e.\ $t \in I$.
\end{lemma}

\begin{proof}
Fix $s,t \in I$ with $s<t$.
We denote by $\Gamma_{[s,t]}$ the set of absolutely continuous curves $\gamma:[s,t] \rightarrow M$
endowed with the uniform (supremum) topology, and define the evaluation map
$e_{\tau}:\Gamma_{[s,t]} \rightarrow M$ at $\tau \in [s,t]$ by $e_{\tau}(\gamma):=\gamma(\tau)$.
By virtue of \cite[Theorem 8.2.1]{AGS}, there exists a probability measure $\Pi \in \mathcal{P}(\Gamma_{[s,t]})$
such that $(e_{\tau})_{\sharp}\Pi =\mu_{\tau}$ for all $\tau \in [s,t]$ and that
$\Pi$ is concentrated on the set of curves $\gamma$ solving $\dot{\gamma}(\tau)=\Phi_{\tau}(\gamma(\tau))$
for a.e.\ $\tau \in [s,t]$.
Since
\[ d^2\big( \gamma(s),\gamma(t) \big) \le (t-s) \int_s^t F^2\big( \gamma(\tau),\dot{\gamma}(\tau) \big) \,d\tau
 =(t-s) \int_s^t F^2\big( \gamma(\tau),\Phi_{\tau}(\gamma(\tau)) \big) \,d\tau \]
holds for $\Pi$-a.e.\ $\gamma$, we see
\[ d_W(\mu_s,\mu_t) \le
 \bigg( \int_{\Gamma_{[s,t]}} d^2\big( \gamma(s),\gamma(t) \big) \,\Pi(d\gamma) \bigg)\!^{1/2}
 \le (t-s)^{1/2} \bigg( \int_s^t F^2_W(\mu_{\tau},\Phi_{\tau}) \,d\tau \bigg)\!^{1/2}. \]
Hence we have $|\dot{\mu}_t| \le F_W(\mu_t,\Phi_t)$ for a.e.\ $t \in I$.
\end{proof}

\begin{thm}\label{th:coeq}
Let $I \subset \R$ be an open interval and $(\mu_t)_{t \in I} \subset \mathcal{P}(M)$
be a locally Lipschitz continuous curve.
Then there exists a Borel vector field $\Phi(t,x) \in T_xM$ on $I \times M$ with
$F(\Phi) \in L^{\infty}_{\mathrm{loc}}(I \times M, d\mu_tdt)$ satisfying the continuity equation $(\ref{eq:coeq})$.
Moreover, such a vector field $\Phi$ is unique up to a difference on a null measure set with respet to $d\mu_t dt$
and satisfies $F_W(\mu_t,\Phi_t) =|\dot{\mu}_t|$ a.e.\ $t \in I$.
\end{thm}

\begin{proof}
Without loss of generality, we assume that $I=(0,1)$ and $(\mu_t)_{t \in I}$ is Lipschitz continuous.
We consider the functional $\Psi$ on the space $V:=\{ D\psi=(D\psi_t)_{t \in I} \,:\, \psi \in \mathcal{C}^{\infty}_c(I \times M) \}$
defined by
\[ \Psi(D\psi) :=-\int_I \int_M \partial_t \psi_t \,d\mu_t \,dt. \]
Clearly $\Psi$ is well-defined and linear.
We equip $V$ with the norm
\[ F^*_V(D\psi) := \bigg( \int_I \int_M F^{*2}\big( x,D\psi_t(x) \big) \,\mu_t(dx) \,dt \bigg)^{1/2}. \]

Given $D\psi \in V$, we see
\begin{align*}
\Psi(D\psi) &= \lim_{\varepsilon \downarrow 0} \int_I \int_M
 \frac{\psi(x,t-\varepsilon)-\psi(x,t)}{\varepsilon} \,\mu_t(dx) \,dt \\
&= \lim_{\varepsilon \downarrow 0} \int_I \frac{1}{\varepsilon} \bigg\{
 \int_M \psi_t \,d\mu_{t+\varepsilon} -\int_M \psi_t \,d\mu_t \bigg\} \,dt.
\end{align*}
Denote by $\pi_{t,t+\varepsilon}$ the optimal coupling of $(\mu_t,\mu_{t+\varepsilon})$.
Taking
\[ \psi_t(y)-\psi_t(x) =F^*\big( x,D\psi_t(x) \big) d(x,y) +o\big( d(x,y) \big) \]
into account, we deduce that
\begin{align*}
\Psi(D\psi) &= \liminf_{\varepsilon \downarrow 0} \int_I \frac{1}{\varepsilon}
 \bigg\{ \int_{M \times M} \{ \psi_t(y)-\psi_t(x) \} \,\pi_{t,t+\varepsilon}(dx\,dy) \bigg\} \,dt \\
&\le \bigg( \int_I \int_M F^{*2}\big( x,D\psi_t(x) \big) \,\mu_t(dx) \,dt \bigg)^{1/2}
 \liminf_{\varepsilon \downarrow 0} \bigg( \int_I \frac{d_W(\mu_t,\mu_{t+\varepsilon})^2}{\varepsilon^2} \,dt \bigg)^{1/2} \\
&= F^*_V(D\psi) \bigg( \int_I |\dot{\mu}_t|^2 \,dt \bigg)^{1/2}.
\end{align*}
We similarly obtain $\Psi(D\psi) \ge -F^*_V(D(-\psi)) (\int_I |\dot{\mu}_t|^2 \,dt)^{1/2}$.
Hence $\Psi$ is a bounded functional and extended to the closure $\overline{V}$ with respect to $F^*_V$.

Thus we find unique $\alpha \in \overline{V}$ (up to a difference on a null measure set)
maximizing the functional $\Psi -F^{*2}_V/2$ on $\overline{V}$.
We set $\Phi_t:=J^*_W(\mu_t,\alpha_t)$ and observe by contruction that, for any $\psi \in \mathcal{C}^{\infty}_c(I \times M)$,
\[ \int_I \int_M \{ \partial_t \psi_t +D\psi_t(\Phi_t) \} \,d\mu_t \,dt=0. \]
This is nothing but the desired continuity equation $(\ref{eq:coeq})$.
Strict convexity of the norm squared $F^{*2}_V/2$ ensures that $\alpha$ is actually a unique element satisfying $(\ref{eq:coeq})$.

Take a sequence $\{ D\psi^{(n)} \}_{n \in \N} \subset V$ converging to $\alpha$.
Then we find
\begin{align*}
F^{*2}_V(\alpha) &= \lim_{n \to \infty} \int_I \int_M D\psi^{(n)}_t(\Phi_t) \,d\mu_t \,dt
 = \lim_{n \to \infty} \Psi(D\psi^{(n)}) \\
&\le \lim_{n \to \infty} F^*_V(D\psi^{(n)}) \bigg( \int_I |\dot{\mu}_t|^2 \,dt \bigg)^{1/2}
 = F^*_V(\alpha) \bigg( \int_I |\dot{\mu}_t|^2 \,dt \bigg)^{1/2}.
\end{align*}
Combining this with Lemma \ref{lm:tanv} shows $F_W(\mu_t,\Phi_t)=|\dot{\mu}_t|$ a.e.\ $t \in I$.
\end{proof}

\begin{defn} \rm
For each locally Lipschitz continuous curve $(\mu_t)_{t \in I} \subset \mathcal{P}(M)$,
we denote by $\dot{\mu}_t \in T_{\mu_t}\mathcal{P}$ its {\it tangent vector field} given by Theorem \ref{th:coeq}.
\end{defn}

\begin{cor}
For any $\mu,\nu \in \mathcal{P}(M)$, we have
\[ d_W(\mu,\nu) = \inf_{(\mu_t)_{t \in [0,1]}} \bigg( \int_0^1 F^2_W(\mu_t,\dot{\mu}_t) \,dt \bigg)^{1/2}, \]
where the infimum is taken over all locally Lipschitz continuous curves $(\mu_t)_{t \in [0,1]} \subset \mathcal{P}(M)$
with $\mu_0=\mu$ and $\mu_1=\nu$.
\end{cor}

\begin{proof}
Recall that $F_W(\mu_t,\dot{\mu}_t)=|\dot{\mu}_t|$ a.e.\ by Theorem \ref{th:coeq}.
Then the inequality $\le$ follows from Lemma \ref{lm:metd},
and equality is attained by a minimal geodesic from $\mu$ to $\nu$.
\end{proof}

\bigskip

For $\mu \in \mathcal{P}(M)$, we define the exponential map $\exp_{\mu}: T_\mu\mathcal{P} \rightarrow \mathcal{P}(M)$ by
$\exp_\mu(\Phi) :=(\exp \Phi)_\sharp \mu$.
Given a function $S$ on (a subset of) $\mathcal{P}(M)$, we say that $S$ is differentiable at $\mu \in \mathcal{P}(M)$
in direction $\Phi \in T_\mu \mathcal{P}$ if the directional derivative
\[ D_\Phi S(\mu):=\lim_{t \downarrow 0} \frac{S(\exp_\mu(t\Phi))-S(\mu)}{t} \]
exists.
We say that $S$ is differentiable at $\mu \in \mathcal{P}(M)$ if there exists
$\alpha \in T_\mu^* \mathcal{P}$ such that $\langle \alpha,\Phi \rangle_\mu =D_\Phi S(\mu)$
holds for all $\Phi =\Nabla \varphi \in T_\mu \mathcal{P}$ with $\varphi \in \mathcal{C}^{\infty}(M)$.
In this case, this $\alpha$ is denoted by $DS(\mu)$ and called the derivative of $S$ at $\mu$.
The gradient vector of $S$ at $\mu$ is defined by $\Nabla_W S(\mu) :=J^*_W(\mu,DS(\mu))$.

\begin{defn}\label{df:Wgrad} \rm
A continuous curve $(\mu_t)_{t \ge 0} \subset \mathcal{P}(M)$ which is
locally Lipschitz continuous on $(0,\infty)$ is called a {\it gradient flow} for $S$ if
$\dot{\mu}_t=\Nabla_W (-S)(\mu_t)$ holds at a.e.\ $t \in (0,\infty)$.
\end{defn}

\begin{prop}\label{pr:Wgrad}
Take $\mu=\rho m \in \mathcal{P}_{\mathrm{ac}}(M)$ such that $\rho \in H^1(M)$.
If $-\log \rho \not\in H^1(M,\mu)$, then $-\mathrm{Ent}$ is not differentiable at $\mu$.
If $-\log \rho \in H^1(M,\mu)$, then $-\mathrm{Ent}$ is differentiable at $\mu$ and
the gradient vector is given by
\[ \Nabla_W (-\mathrm{Ent})(\mu) =\frac{1}{\rho} \Nabla(-\rho) \in T_\mu \mathcal{P}. \]
In particular, its norm squared $F_W^2(\mu,\Nabla_W(-\mathrm{Ent})(\mu))$
coincides with the Fisher information with respect to the reverse Finsler structure $\overleftarrow{F}:$
\[ \overleftarrow{I}(\mu)
 :=\int_M \overleftarrow{F}^2 \big( x,\overleftarrow{\Nabla}\rho(x) \big) \frac{1}{\rho(x)} \,m(dx)
 =\int_M F^2\big( x,\Nabla(-\rho)(x) \big) \frac{1}{\rho(x)} \,m(dx). \]
\end{prop}

\begin{proof}
Fix arbitrary $\varphi \in \mathcal{C}^{\infty}(M)$ and put $\Phi:=\Nabla \varphi$,
$U_0:=\{ x \in M \,:\, \Phi(x)=0 \}$.
By virtue of Lemma \ref{lm:ccon}, the function $t\varphi$ is $d^2/2$-convex for sufficiently small $t>0$.
Hence the map $T_t(x):=\exp_x(t\Phi(x))$ is the unique optimal transport from $\mu$ to $\mu_t:=(T_t)_{\sharp}\mu$.
We will use some properties of $T_t$ and $\mu_t$ established in \cite{Oint}.
The map $T_t$ is injective on a subset of $\mu$-full measure and $\mu_t$ is absolutely continuous,
so that we can write $\mu_t=\rho_t m$.
The map $T_t$ is $\mathcal{C}^{\infty}$ on $M \setminus U_0$ as $T_t(x)$ is not a cut point of $x$.
For $\mu$-a.e.\ $x \in M$, we have the Jacobian equation $\rho(x) =\rho_t(T_t(x)) \mathbf{D}[DT_t(x)]$.
Here $\mathbf{D}[DT_t(x)]$ denotes the Jacobian of the linear operator $DT_t(x):T_xM \rightarrow T_{T_t(x)}M$
with respect to $m$.
That is to say, $\mathbf{D}[DT_t(x)]:=1$ if $x \in U_0$, and
\[ \mathbf{D}[DT_t(x)] :=\frac{m_{T_t(x)}(DT_t(A))}{m_x(A)} \]
for $x \in M \setminus U_0$, where $A \subset T_xM$ is an arbitrary nonempty, bounded open set.

The change of variable formula and the Jacobian equation $\rho =\rho_t(T_t) \mathbf{D}[DT_t]$ show that
\begin{align*}
\mathrm{Ent}(\mu_t) &= \int_M \rho_t \log \rho_t \,dm
 =\int_M \rho_t(T_t) \log \big( \rho_t(T_t) \big) \mathbf{D}[DT_t] \,dm \\
&= \int_M \rho \log\bigg( \frac{\rho}{\mathbf{D}[DT_t]} \bigg) \,dm
 =\mathrm{Ent}(\mu) -\int_M \log(\mathbf{D}[DT_t]) \rho \,dm.
\end{align*}
Thus we have
\begin{align}
\lim_{t \downarrow 0} \frac{\mathrm{Ent}(\mu) -\mathrm{Ent}(\mu_t)}{t}
&=\lim_{t \downarrow 0} \int_M \frac{\log(\mathbf{D}[DT_t])}{t} \rho \,dm
 =\lim_{t \downarrow 0} \int_M \frac{\mathbf{D}[DT_t]-1}{t} \rho \,dm \nonumber\\
&= \lim_{t \downarrow 0} \int_M \frac{\rho -\rho(T_t)}{t} \mathbf{D}[DT_t] \,dm
 =-\int_M D\rho(\Phi) \,dm. \label{eq:dEnt}
\end{align}

If $-\log \rho \in H^1(M,\mu)$, then we obtain $D(-\mathrm{Ent})(\mu) =-(D\rho)/\rho =D(-\log \rho)$ and
\[ \Nabla_W(-\mathrm{Ent})(\mu) =\Nabla(-\log \rho) =\frac{1}{\rho} \Nabla(-\rho). \]
In the other case where $-\log \rho \not\in H^1(M,\mu)$, we approximate $\rho$ by smooth positive $\sigma$
and consider $\varphi=-\log \sigma$.
Then the above calculation $(\ref{eq:dEnt})$ leads
\[ \limsup_{\nu \to \mu} \frac{\mathrm{Ent}(\mu)-\mathrm{Ent}(\nu)}{d_W(\mu,\nu)} =\infty. \]
Hence $-\mathrm{Ent}$ is not differentiable at $\mu$.
\end{proof}

\begin{thm}\label{th:Wgrad}
Let $(\mu_t)_{t \ge 0} \subset \mathcal{P}_{\mathrm{ac}}(M)$ be a continuous curve which is
locally Lipschitz continuous on $(0,\infty)$, and assume that $\mu_t=\rho_t m$ with $\rho_t \in H^1(M)$ a.e.\ $t \in (0,\infty)$.
Then $(\mu_t)_{t \ge 0}$ is a gradient flow for the relative entropy if and only if
$(\rho_t)_{t \ge 0}$ is a heat flow with respect to the reverse Finsler structure $\overleftarrow{F}$ of $F$.
\end{thm}

\begin{proof}
If $(\mu_t)_{t \ge 0}$ is a gradient flow, then Proposition \ref{pr:Wgrad} yields that
\[ \dot{\mu}_t =\Nabla_W(-\mathrm{Ent})(\mu_t) =\frac{\Nabla(-\rho_t)}{\rho_t} \]
for a.e.\ $t \in (0,\infty)$.
Then it follows from the continuity equation $(\ref{eq:coeq})$ that,
for any test function $\psi \in \mathcal{C}_c^{\infty}((0,\infty) \times M)$,
\[ -\int_0^{\infty} \int_M \psi_t \partial_t \rho_t \,dm \,dt
 =\int_0^{\infty} \int_M \partial_t \psi_t \, \rho_t \,dm \,dt
 =-\int_0^{\infty} \int_M D\psi_t\big( \Nabla(-\rho_t) \big) \,dm \,dt. \]
Therefore $-\rho_t$ is a heat flow with respect to $F$ or, equivalently,
$\rho_t$ is a heat flow with respect to $\overleftarrow{F}$.

Conversely, if $(\rho_t)_{t \ge 0}$ is a heat flow with respect to $\overleftarrow{F}$,
then a similar calculation shows that $\{ \Nabla(-\rho_t) \}/\rho_t$ satisfies the continuity equation $(\ref{eq:coeq})$.
We remark that, given $0<t_0<t_1<\infty$, approximating $\rho$ with smooth positive
$\sigma \in C^{\infty}([t_0,t_1] \times M)$ and considering $\psi =-\log \sigma$ yields
\begin{align*}
&\int_{t_0}^{t_1} \int_M \frac{F^2(\Nabla(-\rho_t))}{\rho_t} \,dm \,dt \\
&= \int_{t_0}^{t_1} \int_M \partial_t \rho_t \,dm \,dt
 -\int_M \rho_{t_1} \log \rho_{t_1} \,dm +\int_M \rho_{t_0} \log \rho_{t_0} \,dm \\
&= \mathrm{Ent}(\mu_{t_0}) -\mathrm{Ent}(\mu_{t_1})
 < \infty.
\end{align*}
Therefore $-\log \rho_t \in H^1(M,\mu_t)$ and $\dot{\mu}_t =\{ \Nabla(-\rho_t) \}/\rho_t =\Nabla_W(-\mathrm{Ent})(\mu_t)$
for a.e.\ $t \in (0,\infty)$ by Theorem \ref{th:coeq} and Proposition \ref{pr:Wgrad}.
\end{proof}

\begin{cor}\label{cr:Wheat}
Under the same assumptions as in Theorem $\ref{th:Wgrad}$, the following are equivalent$:$
\begin{description}
\item[(i)]
$(\mu_t)_{t \ge 0}$ is a gradient flow for the relative entropy on the reverse Wasserstein space
$($i.e., the space of probability measures with the reverse Wasserstein distance$);$

\item[(ii)]
$(\mu_t)_{t \ge 0}$ solves the ODE $\dot{\mu}_t =-\Nabla_W \mathrm{Ent}(\mu_t)$ on the Wasserstein space$;$

\item[(iii)]
$(\rho_t)_{t \ge 0}$ solves the heat equation on $M$.
\end{description}
\end{cor}

\begin{remark} \rm
(1) What is missing in Theorem \ref{th:Wgrad} is the contraction property of the heat flow in the Wasserstein space
which is well-known in the Riemannian setting (see, e.g., \cite{vRS} and \cite{Ogra}).
Compare this with Corollary \ref{cr:cont}.
As mentioned in \cite[page 4]{AGS}, even the contraction of gradient flows of ($K$-)convex functions
on Banach spaces is still an open problem.

\medskip

(2)
In Theorem \ref{th:Wgrad}, the existence of the heat flow starting from given $\mu_0$ is guaranteed by
Theorem \ref{th:glob}.
On the other hand, as the relative entropy is $K$-convex if $(M,F,m)$ satisfies
the bound $\infty$-$\mathrm{Ric} \ge K$ (Theorem \ref{th:CD}),
we can argue as in \cite[\S 2]{AGS} or \cite[\S 5]{Ogra}
(except right differentiability for which we need tangent cones) to obtain a continuous curve
$(\mu_t)_{t \ge 0} \in \mathcal{P}(M)$ which satisfies the following properties:
\begin{itemize}
\item[(i)]
The curve $t \mapsto \mu_t$ is locally Lipschitz continuous on $(0,\infty)$.

\item[(ii)]
For all $t>0$, we have
\begin{align*}
&\lim_{\delta \downarrow 0} \frac{d_W(\mu_t,\mu_{t+\delta})}{\delta}
 =F_W\big( \mu_t,\Nabla_W(-\mathrm{Ent})(\mu_t) \big), \\
&\mathrm{Ent}(\mu_t) =\mathrm{Ent}(\mu_0) -\int_0^t F_W\big( \mu_s,\Nabla_W(-\mathrm{Ent})(\mu_s) \big)^2 \,ds.
\end{align*}
\end{itemize}
Thus, in a certain sense $(\mu_t)_{t \ge 0}$ will be a gradient flow for the entropy. However,
it is unclear whether it is actually a gradient flow in the sense of Definition \ref{df:Wgrad}.
\end{remark}

\section{Appendix}

\subsection{Proof of $\partial_tu\in H_0^1$ for global solutions on compact $M$}

Let $u$ be a global solution of the heat equation on a compact space $M$ with $u_0\in H^1_0(M)$.
We know from Theorem \ref{th:glob} that
$v(t,x)=\partial_tu(t,x)$ exists for a.e.\ $(t,x)$ and satisfies
\[ \int_0^T\int_M v^2\,dm\,dt\le\mathcal{E}(u_0). \]
For arbitrary $\delta\in\mathds{R}$ put
$v^{(\delta)}(t,x)=(u(t+\delta,x)-u(t,x))/\delta$.
Then it follows from $(\ref{l2-unique})$ that
\[ \partial_t\|v^{(\delta)}_t\|_{L^2}^2 \le -4\kappa_M \mathcal{E}(v_t^{(\delta)}). \]
Hence, for $|\delta|\le\tau\le T$
\begin{align*}
&4\tau\kappa_M \int_\tau^T\mathcal{E}(v_t^{(\delta)}) \,dt
 \le 4\kappa_M \int_0^T\int_s^T\mathcal{E}(v_t^{(\delta)})\,dt\,ds
 \le \int_0^T \|v^{(\delta)}_s\|_{L^2}^2 \,ds\\
&=\int_M\int_0^T \bigg( \frac1\delta \int_s^{s+\delta} v_t \,dt \bigg)^2 \,ds\,dm
 \le \int_M\int_0^{T+\tau} v_t^2 \,dt\,dm \le \mathcal{E}(u_0).
\end{align*}
Therefore, the family $\{ v^{(\delta)} \,:\, |\delta|\le\tau \}$ is bounded in the norm
\[ \bigg(\int_\tau^T[\mathcal{E}(w_t)+\|w_t\|_{L^2}^2]\,dt \bigg)^{1/2} \]
of $L^2([\tau,T], H^1_0(M))$ for any $\tau>0$.
Reflexivity and completeness of $H^1_0(M)$ then imply the existence of $\tilde v\in
L^2([\tau,T], H^1_0(M))$ such that $v^{(\delta)}\to \tilde v$ in the given norm.
This in particular implies convergence in $L^2$ and thus $\tilde v=v$.
Therefore, $v_t\in H^1_0(M)$ for a.e.\ $t$ with locally square integrable norm of the derivative $F^*(Dv_t)$.
Note that we used the compactness of $M$ only for ensuring $\kappa_M >0$.

\subsection{The same for local solutions on arbitrary $M$}\label{ap:2}

For local solutions, essentially the same arguments apply. For each
open set $\Omega_0$ relatively compact in $\Omega$ we choose another
relatively compact open set $\Omega_1$ containing the closure of
$\Omega_0$ and a (cut-off) function $\psi\in H_0^1(\Omega_1)$
satisfying $0\le\psi\le1$ and $\max\{F^*(D\psi),F^*(-D\psi)\}\le C$ on $M$
(for some constant $C$) and $\psi=1$ on $\Omega_0$.
For instance, we can choose $\psi(x)=d(\Omega_1,\Omega_0\cup\{x\})/d(\Omega_1,\Omega_0)$.

Then a modification of the above calculations yields,
with $\kappa=\kappa_{\Omega_1}$ and $\overline\kappa=(\lambda_{\Omega_1}\lambda_{\Omega_1}^*)^{-1}$
\begin{align*}
&-\frac12\partial_t \bigg[ \int_{\Omega_1} \psi^2 (v^{(\delta)}_t)^2 \,dm \bigg]
 = \frac1{\delta^2} \int_{\Omega_1}
 \big( D(\psi^2u_{t+\delta})-D(\psi^2u_t) \big) (\Nabla u_{t+\delta}-\Nabla u_t) \,dm\\
&= \frac1{\delta^2} \int_{\Omega_1}
 \psi^2 (Du_{t+\delta}-Du_t) (\Nabla u_{t+\delta} -\Nabla u_t) \,dm\\
&\quad +\frac1{\delta^2} \int_{\Omega_1}
 (u_{t+\delta}-u_t) 2\psi D\psi (\Nabla u_{t+\delta} -\Nabla u_t) \,dm\\
&\stackrel{(\ast)} \ge \frac{\kappa}{\delta^2} \int_{\Omega_1}
 \psi^2 F^{*2}(Du_{t+\delta}-Du_t) \,dm \\
&\quad -\frac{2C\overline\kappa}{\delta} \|u_{t+\delta}-u_t\|_{L^2(\Omega_1)} \cdot
 \bigg( \int_{\Omega_1} \psi^2 F^{*2}\bigg( \frac1\delta(Du_{t+\delta}-Du_t) \bigg) \,dm \bigg)^{1/2}\\
&\ge \frac{\kappa}{2} \int_{\Omega_0} F^{*2}(Dv^{(\delta)}_t) \,dm
 -\frac{2C^2\overline\kappa^2}{\kappa} \|v_t^{(\delta)}\|_{L^2(\Omega_1)}^2.
\end{align*}
For the inequality $(\ast)$ we use in addition to the previous argument the fact that
\[ F\big( J^*(\alpha)-J^*(\beta) \big) \le \overline\kappa F^*(\alpha-\beta) \]
which follows from our basic assumption (\ref{eq:F-str2}) on $F$ since for some intermediate point
$\gamma\in T_x^*M$ between $\alpha$ and $\beta$ we have
\[ [J^*(\alpha)-J^*(\beta)] \cdot J[J^*(\alpha)-J^*(\beta)] =g^*(\gamma)(\alpha-\beta) \cdot J[J^*(\alpha)-J^*(\beta)] \]
which implies
\begin{align*}
F^2\big( J^*(\alpha)-J^*(\beta) \big) &\le F^2\big( g^*(\gamma) \cdot (\alpha-\beta) \big) \\
&\le \frac1{\lambda_{\Omega_1}} (\alpha-\beta)^T \cdot g^*(\gamma)^T \cdot g^*(\gamma) \cdot (\alpha-\beta) \\
&\le \frac1{\lambda_{\Omega_1} \lambda^{*2}_{\Omega_1}} |\alpha-\beta|^2
 \le \frac1{\lambda^2_{\Omega_1} \lambda^{*2}_{\Omega_1}} F^{*2}(\alpha-\beta).
\end{align*}
Hence, for $|\delta|\le\tau\le T$
\[ c \int_\tau^T\mathcal{E}_{\Omega_0}(v_t^{(\delta)}) \,dt
 \le \int_0^T \|v^{(\delta)}_s\|_{L^2(\Omega_1)}^2 \,ds
 \le \int_0^{T+\tau} \|v_s\|_{L^2(\Omega_1)}^2 \,ds
 \le \mathcal{E}_{\Omega_1}(u_0) \]
with $c=\kappa\tau(1+4C^2\overline\kappa T/\kappa)^{-1}$.
The same argumentation as before now implies that $v_t\in H^1_{\mathrm{loc}}(\Omega)$
for a.e.\ $t$ provided $u_0\in H^1_{\mathrm{loc}}(\Omega)$.

\subsection{Proof of $u\in H^2_{\mathrm{loc}}$ for local solutions}\label{ap:3}

Let $\Omega_1\subset\Omega$ be an open subset on which a global coordinate system is given.
Fix $k\in\{1,\ldots,n\}$ and put $x^\delta=x+\delta e_k$ for small $\delta\in\mathds{R}$ as well as
$D_k^\delta\varphi(x)=(\varphi(x^\delta)-\varphi(x))/\delta$.
Observe that
\[ D_k^\delta(\varphi \psi)(x)=\varphi(x^\delta) D_k^\delta\psi(x)+\psi(x) D_k^\delta\varphi(x) \]
and
\[ \int_{\Omega_1} \varphi(x) D_k^\delta\psi(x) \,m(dx) =-\int_{\Omega_1} D_k^{-\delta}\varphi(x) \psi(x)\,m(dx) \]
for all compactly supported $\varphi$ and $\psi$ on $\Omega_1$.
If $u$ is a solution to the heat equation then for every test function
$\varphi$ which is compactly supported in $\Omega_1$
\begin{align*}
&-\int (D_k^{-\delta} \varphi) (\partial_tu) e^{-V} \,m(dx) \\
&= \int [D(D_k^{-\delta}\varphi)\cdot J^*(Du)] e^{-V} \,m(dx)
 = -\int D\varphi \cdot D_k^{\delta}\big( J^*(Du) e^{-V} \big) \,m(dx)\\
&= -\int \big[ D\varphi(x) \cdot D_k^{\delta}\big( J^*(Du) \big)(x) \big] e^{-V(x^\delta)} \,m(dx) \\
&\qquad -\int [D\varphi(x)\cdot J^*(Du)(x)] D_k^{\delta}(e^{-V(x)}) \,m(dx).
\end{align*}
On the other hand,
\begin{align*}
&-\int (D_k^{-\delta}\varphi) (\partial_tu) e^{-V} \,m(dx)
 = \int \varphi D_k^{\delta}\left(\partial_tu \cdot e^{-V}\right) \,m(dx)\\
&= \int \varphi(x) \partial_t(D_k^{\delta}u)(x) e^{-V(x)} \,m(dx)
 + \int \varphi(x) \partial_tu(x^{\delta}) D_k^{\delta} (e^{-V(x)}) \,m(dx).
\end{align*}
That is,
\begin{align}
-\int \varphi \partial_t(D_k^{\delta}u) \,dm
&= \int \big[ D\varphi(x) \cdot D_k^{\delta}\big( J^*(Du) \big)(x) \big]
 e^{V(x)-V(x^\delta)} \,m(dx) \nonumber\\
&\quad +\int [D\varphi(x)\cdot J^*(Du)(x)]
 e^{V(x)} D_k^{\delta}(e^{-V(x)}) \,m(dx) \label{h2-weak}\\
&\quad +\int \varphi(x) \partial_tu(x^{\delta}) e^{V(x)} D_k^{\delta}(e^{-V(x)}) \,m(dx). \nonumber
\end{align}

To simplify the presentation, let us first of all treat the
particular case where $u$ is a global solution on $\Omega_1$, i.e. $u\in H^1_0(\Omega_1)$.
This allows to choose $\varphi=D_k^{\delta}u$ which then yields
\begin{align*}
-\frac12 \partial_t \int |D_k^{\delta}u(x)|^2 \,m(dx)
&= \int \big[ D_k^\delta (Du)(x)\cdot D_k^\delta \big( J^*(Du) \big)(x) \big] e^{V(x)-V(x^\delta)} \,m(dx)\\
&\quad +\int [D_k^\delta (Du)(x)\cdot J^*(Du)(x)] e^{V(x)} D_k^\delta(e^{-V(x)}) \,m(dx)\\
&\quad +\int D_k^\delta u(x) \partial_tu(x^\delta) e^{V(x)} D_k^\delta(e^{-V(x)}) \,m(dx).
\end{align*}
We will estimate each of the three terms on the right-hand side from below (or in modulus).
Using the bound  $|D_kV|\le\Lambda$ from our assumption (\ref{smooth Finsler}) we obtain
$|e^VD_k^\delta(e^{-V})| \le (e^{\delta\Lambda}-1)/\delta \le 2\Lambda$
for all sufficiently small $\delta$ and thus we can estimate the second term as follows
\begin{align*}
&\bigg| \int [D_k^\delta (Du)\cdot J^*(Du)] e^V D_k^\delta(e^{-V}) \,dm \bigg| \\
&\le 2\Lambda \bigg( \int F^{*2}(D_k^\delta Du) \,dm \bigg)^{1/2}
 \bigg( \int F^{*2}(Du) \,dm \bigg)^{1/2}\\
&= 4\Lambda \mathcal{E}_{\Omega_1}(D_k^\delta u)^{1/2} \mathcal{E}_{\Omega_1} (u)^{1/2}.
\end{align*}
The third term can be estimated as
\begin{align*}
&\bigg| \int D_k^\delta u(x) \partial_tu(x^\delta) e^{V(x)} D_k^\delta(e^{-V(x)}) \,m(dx) \bigg|\\
&\le 2\Lambda e^{\delta\Lambda/2} \|D_k^\delta u \|_{L^2({\Omega_1})} \|\partial_t u\|_{L^2({\Omega_1})}.
\end{align*}
Finally, using the bound
\begin{align*}
F\big( x,J^*(x,\alpha)-J^*(x^\delta,\alpha) \big)
&\le \frac{1}{\sqrt{\lambda}}|J^*(x,\alpha)-J^*(x^{\delta},\alpha)|
 \le \frac{1}{\sqrt{\lambda}} \frac{\delta \Lambda}{\sqrt{\lambda^*}} |\alpha| \\
&\le \frac{\delta \Lambda}{\lambda \sqrt{\lambda^*}} F^*(x^{\delta},\alpha)
 =: \delta\Lambda' F^*(x^\delta,\alpha)
\end{align*}
for all $\alpha$ from assumptions $(\ref{eq:F-str3})$, $(\ref{smooth Finsler})$
as well as the basic convexity assumption (\ref{eq:2uni}) of the norm
$F^*$ with $\kappa:=\kappa_{\Omega_1}$, the first term (times $e^{\delta\Lambda}$) yields
\begin{align*}
&e^{\delta\Lambda} \int \big[ D_k^\delta (Du)(x) \cdot D_k^\delta \big( J^*(Du) \big)(x) \big]
 e^{V(x)-V(x^\delta)} \,m(dx)\\
&\ge \frac1{\delta^2} \int [Du(x^\delta)-Du(x)] \cdot
 \big[ J^* \big( x^\delta,Du(x^\delta) \big) -J^*\big( x,Du(x) \big) \big] \,m(dx)\\
&\ge \frac1{\delta^2} \int [Du(x^\delta)-Du(x)] \cdot
 \big[ J^*\big( x,Du(x^\delta) \big) -J^* \big( x,Du(x) \big) \big] \,m(dx)\\
&\qquad -\frac{\Lambda'}{\delta} \int F^*\big( x,Du(x^\delta)-Du(x) \big)
 F^*\big( x^\delta,Du(x^\delta) \big) \,m(dx)\\
&\ge \kappa \int F^{*2}\big( x,D_k^\delta u(x) \big) \,m(dx)\\
&\qquad -\Lambda' \int F^*\big( x,D_k^\delta u(x) \big)
 F^*\big( x^\delta,Du(x^\delta) \big) \,m(dx)\\
&\ge 2\kappa \mathcal{E}_{\Omega_1}(D_k^\delta u)
 -2\Lambda' e^{\delta\Lambda/2} \mathcal{E}_{\Omega_1}(D_k^\delta u)^{1/2} \mathcal{E}_{\Omega_1}(u)^{1/2}.
\end{align*}
Summarizing and integrating with respect to $t\in[0,T]$, we obtain
\begin{align*}
&\frac12\|D_k^\delta u_0\|^2_{L^2({\Omega_1})}
 \ge -\frac12\int_0^T \partial_t \bigg[ \int |D_k^{\delta}u_t(x)|^2 \,m(dx) \bigg] \,dt\\
&\ge \kappa\int_0^T \mathcal{E}_{\Omega_1}(D_k^\delta u_t) \,dt
 -C_1 \int_0^T \mathcal{E}_{\Omega_1}(u_t) \,dt
 -C_2 \int_0^T \|D_k^\delta u_t\|_{L^2({\Omega_1})} \|\partial_tu_t\|_{L^2({\Omega_1})} \,dt.
\end{align*}
We know that
$\int_0^T \|\partial_tu_t\|^2_{L^2({\Omega_1})}\,dt=\mathcal{E}_{\Omega_1}(u_0)-\mathcal{E}_{\Omega_1}(u_T)
\le \mathcal{E}_{\Omega_1}(u_0)$
and $\mathcal{E}_{\Omega_1}(u_t)\le\mathcal{E}_{\Omega_1}(u_0)$ for every $t\ge0$.
Moreover,
\begin{align*}
\|D_k^\delta u_t\|_{L^2({\Omega_1})}^2
&= \int_{\Omega_1} \bigg| \frac1\delta\int_0^\delta D_ku_t(x+te_k)\,dt \bigg|^2 e^{-V(x)} \,m(dx)\\
&\le \int_{\Omega_1} |D_ku_t(x)|^2 e^{-V(x)} \,m(dx)
 \le C \mathcal{E}_{\Omega_1}(u_t).
\end{align*}
Hence,
\[ \int_0^T \mathcal{E}_{\Omega_1}(D_k^\delta u_t) \,dt
 \le C' \mathcal{E}_{\Omega_1}(u_0) <\infty \]
uniformly in $\delta$ (provided $|\delta|$ is sufficiently small).
Thus $w_t=D_ku_t=\lim_{\delta\to0}D_k^\delta u_t$ exists in $H^1_0({\Omega_1})$ for a.e.\ $t$
and satisfies $\int_0^T\mathcal{E}_{\Omega_1}(w_t)\,dt<\infty$.

\bigskip

In order to treat the general case, let us now merely assume that $u$ is a local solution.
Given any point in $M$, we find a neighborhood $\Omega_0$ and another relatively compact open set
$\Omega_1$ containing the closure of $\Omega_0$ and admitting a global coordinate system.
We choose a (cut-off) function $\psi \in H_0^1(\Omega_1)$ satisfying $0\le\psi\le1$ and $F^*(D\psi)\le C$
on $M$ (for some constant $C$) and $\psi=1$ on $\Omega_0$.

Now let us put $\varphi=\psi^2 D_k^\delta u$ in (\ref{h2-weak}).
Then all the integrals $\int \cdots \,m(dx)$ in the
previous calculations have to be changed into $\int \cdots \psi^2(x) \,m(dx)$.
In particular, the leading order term will then be of the form
\[ \kappa \int_{\Omega_1} F^{*2} \big( x,D_k^\delta u(x) \big) \psi^2(x) \,m(dx)
 \ge 2\kappa \mathcal{E}_{\Omega_0}(D_k^\delta u).\]
Moreover, due to Leibnitz rule,
two additional terms will show up (from differentiating the first
factor in $\varphi=\psi^2 D_k^\delta u$ with respect to $D$).
However, these terms can easily be estimated in terms of the above
`leading order term', $\mathcal{E}_{\Omega_1}(u)$ and $\| D^{\delta}_k u \|_{L^2(\Omega_1)}$,
cf.\ estimate $(\ast)$ in the previous section.
It finally implies
\[ \int_0^T \mathcal{E}_{\Omega_0}(D_k^\delta u_t) \,dt \le C \mathcal{E}_{\Omega_1}(u_0) <\infty \]
uniformly in $\delta$ and thus $D_ku_t\in H^1_{\mathrm{loc}}(\Omega)$ for a.e.\ $t$.

\end{document}